\documentclass[11pt,twoside]{amsart}
\usepackage[utf8]{inputenc}
\usepackage[T1]{fontenc}
\usepackage[english]{babel}
\usepackage{amsmath, amssymb, amsthm}            
\usepackage{amstext, amsfonts, a4}
\usepackage{graphicx}
\usepackage{hyperref}
\usepackage{color}
\usepackage{stmaryrd}
\usepackage{comment}
\usepackage{tikz}
\usepackage{pgfplots}
\pgfplotsset{compat=1.15}
\usepackage{mathrsfs}
\usetikzlibrary{arrows}

\theoremstyle{plain}
	\newtheorem{Theo}{Theorem}[section] 
	\newtheorem{Prop}[Theo]{Proposition}        
	\newtheorem{Lem}[Theo]{Lemma}            

	\newtheorem{Conj}[Theo]{Conjecture}

\theoremstyle{definition}
	\newtheorem{Def}[Theo]{Definition}

\theoremstyle{remark}
	\newtheorem{Rema}[Theo]{Remark}

\def\NN{{\mathbb N}}    
\def\ZZ{{\mathbb Z}}     
\def\RR{{\mathbb R}}    
\def\QQ{{\mathbb Q}}    
\def\CC{{\mathbb C}}    
\def\HH{{\mathbb H}^2}    
\def\OP{{\mathcal P}}

\title{Connection points on double regular polygons}
\author{Julien Boulanger
}
\address{\flushleft{Centro de Modelamiento matemático,\newline University of Chile,\newline
Beaucheff 851, Santiago, Chile}}
\email{jboulanger@cmm.uchile.cl}
\date{\today}

\begin{document}

\maketitle

\begin{abstract}
In this paper we study connection points on the double regular $n$-gon translation surface, for $n \geq 7$ odd and its staircase model. For $n \neq 9$, we provide a large family of points with coordinates in the trace field that are not connection points. This family includes the central points, and for $n=7$ we conjecture that all the remaining points are connection points. Further, in the case where $n \geq 7$ is a prime number, we provide a constructive proof by exhibiting an explicit separatrix passing through a central point that does not extend to a saddle connection.
\end{abstract}

\section{Introduction and statement of the results}
The present paper lies at the crossroads between geometric topology and number theory. We study properties of the so-called \emph{Hecke groups} and use a \emph{modulo two reduction} method from \cite{Bo73, HMTY} to derive explicit information on periodic directions for a family of translation surfaces (the surfaces affinely equivalent to a double regular $n$-gon, with $n \geq 7$ odd), and use them to study connection points. As a consequence we also obtain an interesting property for billiards in rational polygons: If a trajectory from the center of a regular $n$-gon goes through a vertex, then, if $n \geq 7$ is odd, the reversed trajectory (starting from the center but with the opposite direction) does not necesarily go to a vertex (although it is the case for even $n$, and for $n=3,5$). \newline

Let $n \geq 7$ be an odd integer, and let $X_n$ be the double regular $n$-gon translation surface, obtained from first gluing two regular polygons along a side, then identifying the other sides by pairs, respecting the parallelism. The purpose of this paper is to study \emph{periodic directions} and \emph{connection points} on $X_n$. A direction on a translation surface is (completely) periodic if every geodesic in this direction is either periodic or is a saddle connection, that is, a geodesic between two singularities. A connection point on a translation surface $X$ is by definition a non-singular point of $X$ such that any geodesic from a singularity through this point extends to a saddle connection. Such points have been introduced by P.~Hubert and T.~Schmidt \cite{HS04} who gave a construction of translation surfaces with infinitely generated Veech groups as branched covers over non-periodic connection points.

It is a consequence of a result by C.~McMullen \cite{Mc06} (see also \cite{Bo88} in the setting of interval exchange transformations) that such points exist in a Veech surface whose \emph{trace field} is quadratic: the connection points are exactly the points with coordinates in the trace field (after a natural normalization). However, there is no such result in higher degree, neither concerning connection points nor about infinitely generated Veech groups. More, we do not know any single example of a non-periodic connection point on a translation surface whose trace field has degree three or greater over $\QQ$. In this paper we study one of the easiest family of non-quadratic surfaces, the double regular $n$-gons for $n \geq 7$ odd, whose trace field is of degree $\frac{1}{2} \varphi (2n)$ over $\QQ$, where $\varphi$ is Euler's totient function. It is a consequence of the work of Arnoux and Schmidt \cite{AS09} that for the double regular $n$-gon with $n \geq 7$, there are points with coordinates in the trace field that are not connection points. In the quest of finding non-periodic connection points, the central points of the double $n$-gon (which are not periodic) seem to be good candidates, but the author proves in \cite{Bo20} that central points of the double heptagon and the double nonagon are not connection points. Hovewer, the method presented in \cite{Bo20} does not extend to the case of polygons with more sides. The main purpose of this paper is to provide a new proof of this result, which generalizes to $n \geq 11$. Namely:

\begin{Theo}\label{theo:central_points}
The central points of the double regular $n$-gon are not connection points for $n \geq 7$ odd.
\end{Theo}

In fact, we provide a large family of points whose coordinates lie in the trace field and which are not connection points. Namely,

\begin{Theo}\label{theo:non_connexion_impair}
Let $n \geq 7$ odd. Assume that $n \neq 9$. Let $\lambda = 2 \cos \left( \frac{\pi}{n} \right)$ and let $P_0$ be a point on the staircase model of the double regular $n$-gon whose coordinates\footnote{Obtained from the developing map after setting the development of a singular point to the origin, see Section \ref{sec:staircase_model} and Remark \ref{rema:coordinates}.} are of the form $\frac{1}{N}(x,y)$, where $x,y \in \ZZ[\lambda]$ and $N \geq 1$ is an odd integer. Then $P_0$ is not a connection point.
\end{Theo}

The central points of the double $n$-gon lie in this category, the denominator $N$ being in fact exactly $n$. Theorem \ref{theo:central_points} is then a corollary of Theorem \ref{theo:non_connexion_impair}. Our proof of Theorem \ref{theo:non_connexion_impair} relies on the fact that for $n = 7$ and $n \geq 11$, there is an obstruction \emph{modulo two} for a direction to be periodic in the double regular $n$-gon. In particular, this method does not work for $n=9$, for which the reduction modulo two does not give any relevant information (see Section \ref{sec:Hecke}). However, we already know from \cite{Bo20} that the central points of the double regular nonagon are not connection points.\newline

At this point, it might be tempting to ask, and rightly so, whether there are any non-periodic connection points on the double regular $n$-gon. We conjecture that such points exist for $n=7$:

\begin{Conj}\label{conj:connexion_points_heptagon_first}
There exist non-periodic connection points on the double regular heptagon.
\end{Conj}
A more precise conjecture is formulated in Section \ref{sec:Hecke} (Conjecture \ref{conj:connexion_points_heptagon}). This is due to the fact that, for $n=7$, the obstruction modulo two seems to be the only obstruction for a direction to be periodic (see Conjecture \ref{conj:mod2hepta} for a more precise statement). For $n = 9$ (and experimentally also for $n \geq 11$), this is not true anymore and there seem to be additional obstructions for being a periodic direction. We conjecture instead:

\begin{Conj}
For $n \geq 9$ odd, there are no non-periodic connection points on the double regular $n$-gon.
\end{Conj}

\subsection*{An explicit separatrix for prime $n$.}

As we will shortly see, our proof of Theorem \ref{theo:central_points} is not constructive: it does not give a specific separatrix passing through the central point which does not extend to a saddle connection. In fact, if $n$ is a prime number, one can exhibit such a separatrix using elementary number theory. Namely:

\begin{Theo}\label{theo:prime}
Let $n \geq 7$ be a prime number. We consider the double regular $n$-gon represented on $\RR^2$ with the origin placed at the central point of one of the $n$-gons and with a vertex at the point of coordinates $(1,0)$. Then, the separatrix starting from the point of coordinates $\left(\cos \frac{2\pi}{n}, \sin \frac{2\pi}{n}\right)$ with direction 
\[ (X,Y) = \left(1+2 \cos \left(\frac{2\pi}{n}\right) \left(1 + \cos \frac{\pi}{n}\right), 2 \sin \left(\frac{2\pi}{n}\right)\left(1 - \cos \frac{\pi}{n}\right)\right), \]
passes through the origin (the central point of the right $n$-gon) and does not extend to a saddle connection. 
\end{Theo}

Another interest of this alternative proof is that it gives, for prime $n \geq 7$, an explicit set of directions in the trace field that are not periodic directions. See Proposition \ref{prop:lambda2}
for a more precise formulation.\newline

In light of Theorem \ref{theo:prime}, one could wonder whether the same separatrix extends to a saddle connection if $n$ is not a prime number, and the answer is false: this separatrix extends to a saddle connection for $n = 9$ and $n=15$. However, it could also be checked on the computer that this separatrix does not extend to a saddle connection\footnote{Or, more precisely, its direction is not periodic modulo two; see Section \ref{sec:Hecke}.} for $17 \leq n \leq 199$, independently of $n$ being prime or not.

\subsection*{Remark: billiard trajectories in a regular $n$-gon}
It is well known that translation surfaces are powerful tools to study billiard trajectories, using the unfolding procedure described in \cite{FK36,KaZe}. For a regular $n$-gon, $n \geq 3$ odd, the unfolding procedure gives rise to a necklace surface described for example in \cite{DL18}, which is a degree $n$ cover of the double regular $n$-gon, and for which in particular periodic directions coincide with those on the double regular $n$-gon. As a consequence one can interpret Theorem \ref{theo:central_points} as an anwser to the following question:\newline

\textit{Given $n \geq 3$ and an ideal billiard trajectory on the regular $n$-gon starting at the very center of the polygon, assume the trajectory reaches a vertex. Is it true that the billard trajectory starting from the center of the polygon but in the opposite direction also reaches a vertex ?}\newline

The answer is obviously yes when $n$ is even, by symmetry. For odd $n$, this is equivalent to the fact that the central points of the double regular $n$-gon are (or not) connection points, which we know is true for $n=3,5$, but, from Theorem \ref{theo:central_points}, is false for $n \geq 7$.

\subsection*{Outline of the proof of Theorem \ref{theo:non_connexion_impair}}
In order to prove that a given point is not a connection point, the key idea is to find an explicit separatrix passing through this point whose direction is not periodic. By a result of Veech \cite{Ve89}, periodic directions coincide with eigendirections of parabolic matrices of the so-called Veech group. The Veech group of the \emph{staircase model of the double $n$-gon} for odd $n\geq 3$ is the Hecke group of order $n$, denoted here $H_n$, and hence determining the periodic directions on the double $n$-gon is equivalent to determining the orbit of $\infty \in \partial \HH$ under the action of $H_n$ on the boundary at infinity of the hyperbolic plane. This problem, sometimes reffered to as \emph{Rosen's cusp challenge}, have been studied by various authors, see \cite{Rosen1954, Leu67, Leu74, Bo73, BoRo, SS95, HMTY} among others. In \cite{Bo20}, we were able to certify that certain directions are not in the orbit of $\infty$ by showing that their \emph{next-integer Hecke continued fraction expansion} are eventually periodic, which is equivalent to saying that the direction is fixed by an hyperbolic element of the Veech group. Although this method works well for the double heptagon and the double nonagon, it doesn't seem to help for polygons with more sides as starting from the double hendecagon we were not able to find any hyperbolic direction in the trace field (see also Remark 9 of \cite{HMTY}). In this latter case, it turns out that looking at directions in the trace field \emph{modulo two} gives an easily computable obstruction for a direction to be periodic, see \cite{Bo73,BoRo,HMTY}. We use this obstruction in our proof of Theorem \ref{theo:non_connexion_impair}: given $P_0$ as in the statement of the theorem, we study its orbit under the action of the group of affine diffeomorphisms and we show, using appropriate twists, that every possible reduction modulo two appears in the orbit, and more, corresponds to a dense set of elements in the surface (Proposition \ref{prop:dense_reduction}). This will imply that there are separatrices passing through $P_0$ whose direction is not periodic because they reduce \emph{modulo two} to a direction which is not periodic.

\subsection*{Organization of the paper.} 
We first provide a short introduction to translation surfaces in Section \ref{sec:translation_surfaces}, and we describe the staircase model of the double regular $n$-gon. Next, we recall useful background on Hecke groups in Section \ref{sec:Hecke} and we explain the \emph{obstruction modulo two} for an element to be in $H_n \cdot \infty$, from \cite{Bo73, HMTY}. Then, we prove Theorem \ref{theo:non_connexion_impair} in Section \ref{sec:non_connexion_impair}, and finally Theorem \ref{theo:prime} in Section \ref{sec:prime}.

\subsection*{Acknowledgements.}
The author is grateful to R.~Gutiérrez-Romo, E.~Lanneau and D.~Massart for discussions related to the content of this paper and for comments on a preliminary version of this text. The author thanks the IDEX Université Grenoble Alpes for funding his mobility grant. This work was supported by Centro de Modelamiento Matemático (CMM) BASAL fund FB210005 for center of excellence from ANID-Chile.

\section{Translation surfaces and their Veech groups}\label{sec:translation_surfaces}
We begin with a short background on translation surfaces and their Veech groups, mainly following \cite{Bo20}. The interested reader could also check out the books \cite{DHV_book, AM_book}.

A (finite) translation surface is the data of a (finite) collection of polygons embedded in $\CC$ and a pairing of the sides, where each side pairing must respect the length and the parallelism. From this data, translation surfaces inherit a flat metric and a finite number of singularities, whose angles are integer multiples of $2\pi$. In particular, geodesics are piecewise straight lines (that may change direction only at singularities), and we can make the following definitions:

\begin{Def}
A \emph{separatrix} is a geodesic line starting at a singularity.\newline
A \emph{saddle connection} is a geodesic line starting and ending at a singularity.\newline
A \emph{connection point} on a translation surface $X$ is a non-singular point $P$ of $X$ such that every separatrix passing through $P$ extends to a saddle connection.
\end{Def}

As said in the introduction, connection points on translation surfaces have been introduced by P.~Hubert and T.~Schmidt \cite{HS04} in order to construct translation surfaces with infinitely generated Veech groups.\newline

\paragraph{\textbf{Moduli space and Veech groups}}
The moduli space of translation surfaces is defined as the set of equivalence classes of translation surfaces up to the "cut-and-paste" equivalence relation. Namely, two translation surfaces represented by polygons and gluings are considered equivalent if one can cut the polygons of the first translation surface and rearrange the pieces (respecting the identifications) to obtain the second translation surface.

The action of $GL_2^+(\RR)$ on $\RR^2$ provides an action on the polygons, which in turn induces an action on the moduli space of translation surfaces. Two surfaces are said \emph{affinely equivalent} if they lie in the same orbit, and the stabilizer of a given translation surface $X$ is called the \emph{Veech group} of $X$, which we will denote by $SL(X)$. In particular, affinely equivalent surfaces have conjugated Veech groups. To an element $M$ of the Veech group of $X$ corresponds (at least) an \emph{affine diffeomorphism} of the surface obtained from applying $M$ and then cutting and pasting back to the original surface using translations\footnote{There may be several ways to cut and paste back to the original surface, but it is not the case for the surfaces we study here.}. Veech groups are named after W. Veech, who showed that they are discrete subgroups of $SL_2(\RR)$ and studied some of their properties: one of their main features is that they capture geometric and dynamical information about the geodesic (directional) flow, namely

\begin{Theo}[\cite{Ve89}]\label{theo:Veech}
Let $X$ be a translation surface such that $SL(X)$ is a lattice. Then for every direction $\theta \in \mathbb{S}^1$, the directionnal flow in direction $\theta$ is either completely periodic (every trajectory is either periodic or a saddle connection) or uniquely ergodic. 

Further, the first case arises if and only if $\theta$ is the eigendirection of a parabolic matrix of $SL(X)$.
\end{Theo}

Because of this result, translation surfaces whose Veech group is a lattice are usually referred to as \emph{Veech surfaces}. A consequence of this result is that in every periodic direction, a Veech surface is decomposed into \emph{cylinders}: A horizontal cylinder $C(\omega,h)$ is an euclidean annulus of the form $[0, \omega] \times (0,h)$ where the right and left boundary are identified: $\forall x \in (0,h), (0,x) \sim (\omega,x)$. The parameters $\omega$ and $h$ are respectively called the width and the height of the cylinder $C(\omega,h)$, and the \emph{modulus} of the cylinder is defined as the ratio $\omega/h$. Additionally, the vector $(\omega,0)$ can be reffered to as the \emph{holonomy vector} of a core curve of the cylinder $C$. Now, an open subset $C$ of the translation surface $X$ is a cylinder in the direction $\theta$ if there exist parameters $\omega,h$ such that $e^{i\theta}C$ is isomorphic to $C(\omega,h)$. The holonomy vector of a core curve of $C$ has now the form $(\omega_C^x,\omega_C^y) := (\omega \cos \theta, \omega \sin \theta)$. A decomposition into cylinders is a finite partition of the surface into (parallel) cylinders, except a finite number of saddle connections.

It is in general a difficult question to characterize explicitly the parabolic limit set of a given lattice fuschian group, and hence to characterize the set of periodic directions for a given Veech surface. It is a classical result that, after a normalization, periodic directions on $X$ have slope belonging to the \emph{trace field} of $SL(X)$, which is the number field generated by the traces of the elements of $SL(X)$. It has been shown by Kenyon and Smillie \cite{KS00} that the trace field has degree at most $g$, the genus of $X$, over $\QQ$. And it is a consequence of the work of C.~McMullen \cite{Mc04b,Mc06} that, if the trace field of $SL(X)$ is quadratic over $\QQ$, then the converse of the above is true: every direction in the trace field is a periodic direction (or equivalently an eigendirection of a parabolic matrix of the Veech group of $X$). See also \cite{Bo88} in the setting of interval exchange transformations. However, the converse is not true in general, and very little is known for periodic directions in translation surfaces whose trace field has degree three or more over $\QQ$. As a consequence, we do not know a single example of a non-periodic\footnote{A non singular point on a translation surface is \emph{periodic} if its orbit under the group of affine diffeomorphisms is finite, see the last paragraph of this section for details.} connection point. The double regular $n$-gon for $n \geq 7$, which we now introduce, is a natural example of Veech surface for which we can try to understand periodic directions.

\subsection{The double regular $n$-gon and its staircase model}\label{sec:staircase_model}

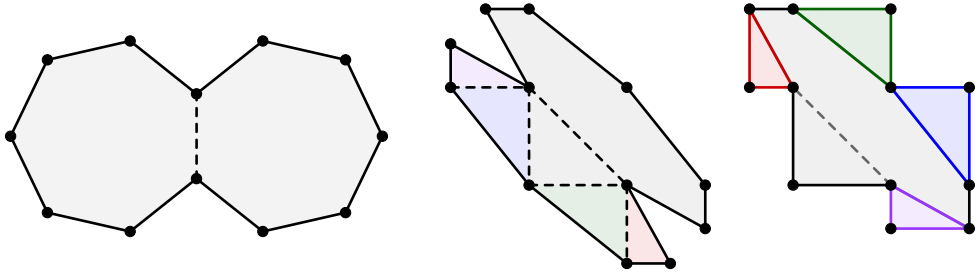
\begin{figure}
\center
\definecolor{ccqqqq}{rgb}{0.8,0,0}
\definecolor{qqwuqq}{rgb}{0,0.39215686274509803,0}
\definecolor{qqqqff}{rgb}{0,0,1}
\definecolor{zzttff}{rgb}{0.6,0.2,1}
\definecolor{wwwwww}{rgb}{0.4,0.4,0.4}
\begin{tikzpicture}[line cap=round,line join=round,>=triangle 45,x=1cm,y=1cm,scale=1.3]
\clip(-5.5,-1) rectangle (5,2);
\fill[line width=1pt,fill=black,fill opacity=0.05] (-1.5,0.5) -- (-1.8765101981412664,1.28183148246803) -- (-2.7225209339563143,1.474927912181824) -- (-3.400968867902419,0.9338837391175587) -- (-3.400968867902419,0.06611626088244243) -- (-2.722520933956315,-0.4749279121818234) -- (-1.8765101981412668,-0.2818314824680298) -- cycle;
\fill[line width=1pt,fill=black,fill opacity=0.05] (-3.400968867902419,0.06611626088244243) -- (-3.400968867902419,0.9338837391175587) -- (-4.079416801848524,1.474927912181824) -- (-4.925427537663571,1.2818314824680304) -- (-5.301937735804838,0.5) -- (-4.925427537663571,-0.281831482468029) -- (-4.079416801848524,-0.47492791218182295) -- cycle;
\fill[line width=1pt,color=wwwwww,fill=wwwwww,fill opacity=0.1] (0,1) -- (1,0) -- (1.8019377358048383,-0.4450418679126291) -- (1.8019377358048383,0) -- (1,1) -- (0,1.8019377358048383) -- (-0.4450418679126291,1.8019377358048383) -- cycle;
\fill[line width=1pt,color=zzttff,fill=zzttff,fill opacity=0.1] (-0.8019377358048383,1.445041867912629) -- (-0.8019377358048383,1) -- (0,1) -- cycle;
\fill[line width=1pt,color=qqqqff,fill=qqqqff,fill opacity=0.1] (-0.8019377358048383,1) -- (0,1) -- (0,0) -- cycle;
\fill[line width=1pt,color=wwwwww,fill=wwwwww,fill opacity=0.1] (0,1) -- (0,0) -- (1,0) -- cycle;
\fill[line width=1pt,color=qqwuqq,fill=qqwuqq,fill opacity=0.1] (0,0) -- (1,0) -- (1,-0.8019377358048383) -- cycle;
\fill[line width=1pt,color=ccqqqq,fill=ccqqqq,fill opacity=0.1] (1,0) -- (1,-0.8019377358048383) -- (1.445041867912629,-0.8019377358048383) -- cycle;
\fill[line width=1pt,color=wwwwww,fill=wwwwww,fill opacity=0.1] (3.7,0) -- (2.7,1) -- (2.254958132087371,1.8019377358048383) -- (2.7,1.8019377358048383) -- (3.7,1) -- (4.501937735804838,0) -- (4.501937735804838,-0.4450418679126291) -- cycle;
\fill[line width=1pt,color=ccqqqq,fill=ccqqqq,fill opacity=0.1] (2.254958132087371,1.8019377358048383) -- (2.254958132087371,1) -- (2.7,1) -- cycle;
\fill[line width=1pt,color=wwwwww,fill=wwwwww,fill opacity=0.1] (2.7,1) -- (2.7,0) -- (3.7,0) -- cycle;
\fill[line width=1pt,color=qqwuqq,fill=qqwuqq,fill opacity=0.1] (2.7,1.8019377358048383) -- (3.7,1.8019377358048383) -- (3.7,1) -- cycle;
\fill[line width=1pt,color=qqqqff,fill=qqqqff,fill opacity=0.1] (3.7,1) -- (4.501937735804838,1) -- (4.501937735804838,0) -- cycle;
\fill[line width=1pt,color=zzttff,fill=zzttff,fill opacity=0.1] (3.7,0) -- (3.7,-0.4450418679126291) -- (4.501937735804838,-0.4450418679126291) -- cycle;
\draw [line width=1pt] (-1.8765101981412664,1.28183148246803)-- (-2.7225209339563143,1.474927912181824);
\draw [line width=1pt] (-2.7225209339563143,1.474927912181824)-- (-3.400968867902419,0.9338837391175587);
\draw [line width=1pt] (-3.400968867902419,0.06611626088244243)-- (-2.722520933956315,-0.4749279121818234);
\draw [line width=1pt] (-2.722520933956315,-0.4749279121818234)-- (-1.8765101981412668,-0.2818314824680298);
\draw [line width=1pt] (-1.8765101981412668,-0.2818314824680298)-- (-1.5,0.5);
\draw [line width=1pt] (-3.400968867902419,0.9338837391175587)-- (-4.079416801848524,1.474927912181824);
\draw [line width=1pt] (-4.079416801848524,1.474927912181824)-- (-4.925427537663571,1.2818314824680304);
\draw [line width=1pt] (-4.925427537663571,1.2818314824680304)-- (-5.301937735804838,0.5);
\draw [line width=1pt] (-1.5,0.5)-- (-1.86,1.26);
\draw [line width=1pt] (-5.301937735804838,0.5)-- (-4.925427537663571,-0.281831482468029);
\draw [line width=1pt] (-4.925427537663571,-0.281831482468029)-- (-4.079416801848524,-0.47492791218182295);
\draw [line width=1pt] (-4.079416801848524,-0.47492791218182295)-- (-3.400968867902419,0.06611626088244243);
\draw [line width=1pt,dash pattern=on 3pt off 3pt] (0,1)-- (1,0);
\draw [line width=1pt,dash pattern=on 3pt off 3pt] (-3.4,0.95)-- (-3.4,0.05);
\draw [line width=1pt] (1,0)-- (1.8019377358048383,-0.4450418679126291);
\draw [line width=1pt] (1.8019377358048383,-0.4450418679126291)-- (1.8019377358048383,0);
\draw [line width=1pt] (1.8019377358048383,0)-- (1,1);
\draw [line width=1pt] (1,1)-- (0,1.8019377358048383);
\draw [line width=1pt] (0,1.8019377358048383)-- (-0.4450418679126291,1.8019377358048383);
\draw [line width=1pt] (-0.4450418679126291,1.8019377358048383)-- (0,1);
\draw [line width=1pt] (0,1)-- (-0.8019377358048383,1.445041867912629);
\draw [line width=1pt] (-0.8019377358048383,1.445041867912629)-- (-0.8019377358048383,1);
\draw [line width=1pt] (-0.8019377358048383,1)-- (0,0);
\draw [line width=1pt] (0,0)-- (1,-0.8019377358048383);
\draw [line width=1pt] (1,-0.8019377358048383)-- (1.445041867912629,-0.8019377358048383);
\draw [line width=1pt] (1.445041867912629,-0.8019377358048383)-- (1,0);
\draw [line width=1pt,dash pattern=on 3pt off 3pt] (-0.8019377358048383,1)-- (0,1);
\draw [line width=1pt,dash pattern=on 3pt off 3pt] (0,1)-- (0,0);
\draw [line width=1pt,dash pattern=on 3pt off 3pt] (0,0)-- (1,0);
\draw [line width=1pt,dash pattern=on 3pt off 3pt] (1,0)-- (1,-0.8019377358048383);
\draw [line width=1pt,color=wwwwww] (2.7,1)-- (2.254958132087371,1.8019377358048383);
\draw [line width=1pt] (2.254958132087371,1.8019377358048383)-- (2.7,1.8019377358048383);
\draw [line width=1pt,color=wwwwww] (2.7,1.8019377358048383)-- (3.7,1);
\draw [line width=1pt,color=wwwwww] (3.7,1)-- (4.501937735804838,0);
\draw [line width=1pt] (4.501937735804838,0)-- (4.501937735804838,-0.4450418679126291);
\draw [line width=1pt,color=wwwwww] (4.501937735804838,-0.4450418679126291)-- (3.7,0);
\draw [line width=1pt,color=ccqqqq] (2.254958132087371,1.8019377358048383)-- (2.254958132087371,1);
\draw [line width=1pt,color=ccqqqq] (2.254958132087371,1)-- (2.7,1);
\draw [line width=1pt,color=ccqqqq] (2.7,1)-- (2.254958132087371,1.8019377358048383);
\draw [line width=1pt] (2.7,1)-- (2.7,0);
\draw [line width=1pt] (2.7,0)-- (3.7,0);
\draw [line width=1pt,dash pattern=on 3pt off 3pt,color=wwwwww] (3.7,0)-- (2.7,1);
\draw [line width=1pt,color=qqwuqq] (2.7,1.8019377358048383)-- (3.7,1.8019377358048383);
\draw [line width=1pt,color=qqwuqq] (3.7,1.8019377358048383)-- (3.7,1);
\draw [line width=1pt,color=qqwuqq] (3.7,1)-- (2.7,1.8019377358048383);
\draw [line width=1pt,color=qqqqff] (3.7,1)-- (4.501937735804838,1);
\draw [line width=1pt,color=qqqqff] (4.501937735804838,1)-- (4.501937735804838,0);
\draw [line width=1pt,color=qqqqff] (4.501937735804838,0)-- (3.7,1);
\draw [line width=1pt,color=zzttff] (3.7,0)-- (3.7,-0.4450418679126291);
\draw [line width=1pt,color=zzttff] (3.7,-0.4450418679126291)-- (4.501937735804838,-0.4450418679126291);
\draw [line width=1pt,color=zzttff] (4.501937735804838,-0.4450418679126291)-- (3.7,0);
\begin{scriptsize}
\draw [fill=black] (-1.5,0.5) circle (1.5pt);
\draw [fill=black] (-1.8765101981412664,1.28183148246803) circle (1.5pt);
\draw [fill=black] (-2.7225209339563143,1.474927912181824) circle (1.5pt);
\draw [fill=black] (-3.400968867902419,0.9338837391175587) circle (1.5pt);
\draw [fill=black] (-3.400968867902419,0.06611626088244243) circle (1.5pt);
\draw [fill=black] (-2.722520933956315,-0.4749279121818234) circle (1.5pt);
\draw [fill=black] (-1.8765101981412668,-0.2818314824680298) circle (1.5pt);
\draw [fill=black] (-4.079416801848524,1.474927912181824) circle (1.5pt);
\draw [fill=black] (-4.925427537663571,1.2818314824680304) circle (1.5pt);
\draw [fill=black] (-5.301937735804838,0.5) circle (1.5pt);
\draw [fill=black] (-4.925427537663571,-0.281831482468029) circle (1.5pt);
\draw [fill=black] (-4.079416801848524,-0.47492791218182295) circle (1.5pt);
\draw [fill=black] (0,0) circle (1.5pt);
\draw [fill=black] (1,0) circle (1.5pt);
\draw [fill=black] (0,1) circle (1.5pt);
\draw [fill=black] (1,1) circle (1.5pt);
\draw [fill=black] (1.8019377358048383,0) circle (1.5pt);
\draw [fill=black] (0,1.8019377358048383) circle (1.5pt);
\draw [fill=black] (-0.4450418679126291,1.8019377358048383) circle (1.5pt);
\draw [fill=black] (1.8019377358048383,-0.4450418679126291) circle (1.5pt);
\draw [fill=black] (2.7,0) circle (1.5pt);
\draw [fill=black] (3.7,0) circle (1.5pt);
\draw [fill=black] (2.7,1) circle (1.5pt);
\draw [fill=black] (3.7,1) circle (1.5pt);
\draw [fill=black] (4.501937735804838,0) circle (1.5pt);
\draw [fill=black] (4.501937735804838,1) circle (1.5pt);
\draw [fill=black] (3.7,1.8019377358048383) circle (1.5pt);
\draw [fill=black] (2.7,1.8019377358048383) circle (1.5pt);
\draw [fill=black] (2.254958132087371,1) circle (1.5pt);
\draw [fill=black] (2.254958132087371,1.8019377358048383) circle (1.5pt);
\draw [fill=black] (3.7,-0.4450418679126291) circle (1.5pt);
\draw [fill=black] (4.501937735804838,-0.4450418679126291) circle (1.5pt);
\draw [fill=black] (-0.8019377358048383,1) circle (1.5pt);
\draw [fill=black] (1,-0.8019377358048383) circle (1.5pt);
\draw [fill=black] (1.445041867912629,-0.8019377358048383) circle (1.5pt);
\draw [fill=black] (-0.8019377358048383,1.445041867912629) circle (1.5pt);
\end{scriptsize}
\end{tikzpicture}
\caption{From the double regular heptagon to its staircase model.}
\label{fig:heptagon_and_staircase}
\end{figure}

Given $n \geq 3$ odd, we consider the double regular $n$-gon $X_n$, made from two copies of a regular $n$-gon glued along a side, and whose other sides are identified by pairs. From $X_n$ one can apply the matrix
\[ 
P = \frac{1}{\sin \frac{(n-1)\pi}{2n}} \begin{pmatrix} \sin \frac{\pi}{n} & - \cos \frac{\pi}{n} + 1 \\ \sin \frac{\pi}{n} & \cos \frac{\pi}{n} + 1 \end{pmatrix}.
\]

and then cut and paste the obtained polygonal representation to obtain a staircase surface $S_n$. This construction is detailled in \cite{Ho12}. See Figure \ref{fig:heptagon_and_staircase} in the case $n=7$.

More precisely, the staircase model $S_n$ can be made of a chain of rectangles $R_i$, $1 \leq i \leq n-2$ where the sides of $R_i$ have length 
\[ \frac{\sin\frac{i\pi}{n}}{\sin \frac{(n-1)\pi}{2n}} \text{ and } \frac{\sin\frac{(i+1)\pi}{n}}{\sin\frac{(n-1)\pi}{2n}},
\]
and are respectively identified to the sides of $R_{i-1}$ and $R_{i+1}$ (the small sides of $R_1$ (resp. $R_{n-2}$) are identified together). There are two advantages of normalizing\footnote{One could also use the normalization by $\sin(\frac{\pi}{n})$ (and we implictly use it in the proof of Theorem \ref{theo:prime}), but it will be more convenient for the proof of Theorem \ref{theo:non_connexion_impair} to normalize by $\sin \left(\frac{(n-1)\pi}{2n}\right)$.} by $\sin \frac{(n-1)\pi}{2n}$: first, $R_{\frac{n-1}{2}}$ is a unit square, but also the lengths of the sides of $R_i$ can then be expressed as a polynomial in $\lambda = \lambda_n := 2 \cos \frac{\pi}{n}$ (we will often omit the index $n$). The latter is a consequence of the following well known fact, which can be proved by a direct computation (see for example \cite[Proposition 2.4]{Da13}):

\begin{Prop}\label{prop:cylinders}
The staircase model or the regular $n$-gon has a horizontal (resp. vertical) cylinder decomposition with $g = \frac{n-1}{2}$ cylinders. Further, every cylinder has modulus $\lambda$.
\end{Prop}

In particular, using that $R_{\frac{n-1}{2}}$ is a unit square, one can express the coordinates of every side as a polynomial in $\lambda$: for $0 \leq i < \frac{n-1}{2}$ the rectangle $R_{\frac{n-1}{2} \pm i}$ has sides of length $Q_i(\lambda)$ and $Q_{i+1}(\lambda)$ where
\[ Q_0(X) = Q_1(X) = 1 \text{ and, for }i \geq 1, Q_{i+1}(X) = XQ_i(X) - Q_{i-1}(X). \]

As a consequence, if we consider the coordinates on $S_n$ given by the embedding in $\CC$ of the staircase model where the origin is set at the bottom left corner of $R_{\frac{n-1}{2}}$ (see Figure \ref{fig:coordinates_staircase}), then every vertex has coordinates in $\ZZ[\lambda]$, and the image of the right central point of the double regular $n$-gon is a barycenter of $n$ vertices, so that its coordinates are of the form $\frac{1}{n}(x,y)$, with $x,y \in \ZZ[\lambda]$.

\begin{Rema}\label{rema:coordinates}
Although it will turn out to be convenient to work with the coordinates obtained from the plane template of Figure \ref{fig:coordinates_staircase}, one could choose other coordinates obtained from the developing map of $S_n$ after fixing the image of a singularity at the origin, and Theorem \ref{theo:non_connexion_impair} would still hold.
\end{Rema}

Further, the horizontal cylinder $C_i$ defined for $1 \leq i \leq \frac{n-1}{2}$ as the union $R_{\frac{n-1}{2}+i-1} \cup R_{\frac{n-1}{2}+i}$ if $i$ is odd (setting $R_{n-1} = \varnothing$ if $\frac{n-1}{2}$ is odd) and the union $R_{\frac{n-1}{2}-i} \cup R_{\frac{n-1}{2}-i+1}$ if $i$ is even (setting $R_0 = \varnothing$ if $\frac{n-1}{2}$ is even) has width
\begin{equation*}
\omega_i = \left\{ 
\begin{array}{ll} 
	Q_{i-1}(\lambda) + Q_{i+1}(\lambda) &\text{ if } 0 \leq i \leq \frac{n-1}{2}-2 \\
	Q_{\frac{n-1}{2} -1}(\lambda)  &\text{ if } i = \frac{n-1}{2}
\end{array}
\right.
\end{equation*}
In particular, we have:
\begin{Lem}\label{lem:widths_cylinders}
For $1 \leq i \leq \frac{n-1}{2}-1$, $\omega_i$ can be expressed as a monic polynomial in $\lambda$ of degree $i$.
\end{Lem}

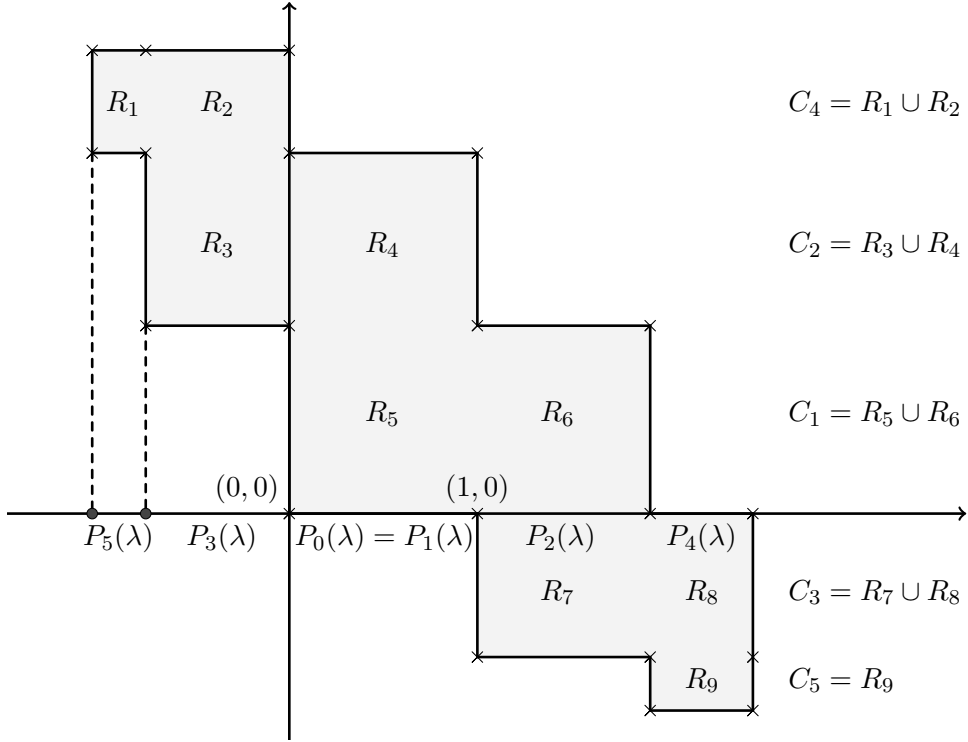
\begin{figure}
\center
\definecolor{uuuuuu}{rgb}{0.26666666666666666,0.26666666666666666,0.26666666666666666}
\begin{tikzpicture}[line cap=round,line join=round,>=triangle 45,x=2.5cm,y=2.5cm]
\clip(-1.5,-1.2) rectangle (3.8,2.9);
\fill[line width=1pt,fill=black,fill opacity=0.05] (0,0) -- (1,0) -- (1,-0.7635211184333672) -- (1.9189859472289947,-0.7635211184333672) -- (1.9189859472289947,-1.0481507949799365) -- (2.4651862966861966,-1.0481507949799365) -- (2.4651862966861966,-0.7635211184333672) -- (2.4651862966861966,0) -- (1.9189859472289947,0) -- (1.9189859472289947,1) -- (1,1) -- (1,1.9189859472289947) -- (0,1.9189859472289947) -- (0,2.4651862966861966) -- (-0.7635211184333672,2.4651862966861966) -- (-1.0481507949799365,2.4651862966861966) -- (-1.0481507949799365,1.9189859472289947) -- (-0.7635211184333672,1.9189859472289947) -- (-0.7635211184333672,1) -- (0,1) -- cycle;
\draw [line width=1pt] (-1.0481507949799365,2.4651862966861966)-- (-0.7635211184333672,2.4651862966861966);
\draw [line width=1pt] (-0.7635211184333672,2.4651862966861966)-- (0,2.4651862966861966);
\draw [line width=1pt] (0,2.4651862966861966)-- (0,1.9189859472289947);
\draw [line width=1pt] (0,1.9189859472289947)-- (1,1.9189859472289947);
\draw [line width=1pt] (1,1.9189859472289947)-- (1,1);
\draw [line width=1pt] (1,1)-- (1.9189859472289947,1);
\draw [line width=1pt] (1.9189859472289947,1)-- (1.9189859472289947,0);
\draw [line width=1pt] (1.9189859472289947,0)-- (2.4651862966861966,0);
\draw [line width=1pt] (2.4651862966861966,0)-- (2.4651862966861966,-0.7635211184333672);
\draw [line width=1pt] (2.4651862966861966,-0.7635211184333672)-- (2.4651862966861966,-1.0481507949799365);
\draw [line width=1pt] (2.4651862966861966,-1.0481507949799365)-- (1.9189859472289947,-1.0481507949799365);
\draw [line width=1pt] (1.9189859472289947,-1.0481507949799365)-- (1.9189859472289947,-0.7635211184333672);
\draw [line width=1pt] (1.9189859472289947,-0.7635211184333672)-- (1,-0.7635211184333672);
\draw [line width=1pt] (1,-0.7635211184333672)-- (1,0);
\draw [line width=1pt] (1,0)-- (0,0);
\draw [line width=1pt] (0,0)-- (0,1);
\draw [line width=1pt] (0,1)-- (-0.7635211184333672,1);
\draw [line width=1pt] (-0.7635211184333672,1)-- (-0.7635211184333672,1.9189859472289947);
\draw [line width=1pt] (-0.7635211184333672,1.9189859472289947)-- (-1.0481507949799365,1.9189859472289947);
\draw [line width=1pt] (-1.0481507949799365,1.9189859472289947)-- (-1.0481507949799365,2.4651862966861966);
\draw [line width=1pt,-to] (-1.5,0)-- (3.6,0);
\draw [line width=1pt,-to] (0,-1.5)-- (0,2.7223638607832705);
\draw (-0.025,0) node[anchor=north west] {$P_0(\lambda) = P_1(\lambda)$};
\draw (1.2,0) node[anchor=north west] {$P_2(\lambda)$};
\draw (-0.6,0) node[anchor=north west] {$P_3(\lambda)$};
\draw (1.95,0) node[anchor=north west] {$P_4(\lambda)$};
\draw (-1.15,0) node[anchor=north west] {$P_5(\lambda)$};
\draw [line width=1pt,dash pattern=on 3pt off 3pt] (-1.0481507949799365,1.9189859472289947)-- (-1.0481507949799365,0);
\draw [line width=1pt,dash pattern=on 3pt off 3pt] (-0.7635211184333672,1)-- (-0.7635211184333671,0);
\draw (0.35,0.65) node[anchor=north west] {$R_5$};
\draw (0.35,1.55) node[anchor=north west] {$R_4$};
\draw (-0.53,1.55) node[anchor=north west] {$R_3$};
\draw (-0.53,2.3) node[anchor=north west] {$R_2$};
\draw (-1.03,2.3) node[anchor=north west] {$R_1$};
\draw (1.28,0.65) node[anchor=north west] {$R_6$};
\draw (1.28,-0.3) node[anchor=north west] {$R_7$};
\draw (2.05,-0.77) node[anchor=north west] {$R_9$};
\draw (2.05,-0.3) node[anchor=north west] {$R_8$};
\draw (2.6,0.65) node[anchor=north west] {$C_1 = R_5 \cup R_6$};
\draw (2.6,1.55) node[anchor=north west] {$C_2 = R_3 \cup R_4$};
\draw (2.6,2.3) node[anchor=north west] {$C_4 = R_1 \cup R_2$};
\draw (2.6,-0.3) node[anchor=north west] {$C_3 = R_7 \cup R_8$};
\draw (2.6,-0.77) node[anchor=north west] {$C_5 = R_9$};
\draw (0,0) node[anchor=south east] {$(0,0)$};
\draw (1,0) node[above] {$(1,0)$};
\begin{scriptsize}
\draw [color=black] (0,0)-- ++(-2pt,-2pt) -- ++(4pt,4pt) ++(-4pt,0) -- ++(4pt,-4pt);
\draw [color=black] (1,0)-- ++(-2pt,-2pt) -- ++(4pt,4pt) ++(-4pt,0) -- ++(4pt,-4pt);
\draw [color=black] (0,1)-- ++(-2pt,-2pt) -- ++(4pt,4pt) ++(-4pt,0) -- ++(4pt,-4pt);
\draw [color=black] (1,1)-- ++(-2pt,-2pt) -- ++(4pt,4pt) ++(-4pt,0) -- ++(4pt,-4pt);
\draw [color=black] (1.9189859472289947,1)-- ++(-2pt,-2pt) -- ++(4pt,4pt) ++(-4pt,0) -- ++(4pt,-4pt);
\draw [color=black] (1,1.9189859472289947)-- ++(-2pt,-2pt) -- ++(4pt,4pt) ++(-4pt,0) -- ++(4pt,-4pt);
\draw [color=black] (1.9189859472289947,0)-- ++(-2pt,-2pt) -- ++(4pt,4pt) ++(-4pt,0) -- ++(4pt,-4pt);
\draw [color=black] (0,1.9189859472289947)-- ++(-2pt,-2pt) -- ++(4pt,4pt) ++(-4pt,0) -- ++(4pt,-4pt);
\draw [color=black] (-0.7635211184333672,1.9189859472289947)-- ++(-2pt,-2pt) -- ++(4pt,4pt) ++(-4pt,0) -- ++(4pt,-4pt);
\draw [color=black] (-0.7635211184333672,1)-- ++(-2pt,-2pt) -- ++(4pt,4pt) ++(-4pt,0) -- ++(4pt,-4pt);
\draw [color=black] (1,-0.7635211184333672)-- ++(-2pt,-2pt) -- ++(4pt,4pt) ++(-4pt,0) -- ++(4pt,-4pt);
\draw [color=black] (1.9189859472289947,-0.7635211184333672)-- ++(-2pt,-2pt) -- ++(4pt,4pt) ++(-4pt,0) -- ++(4pt,-4pt);
\draw [color=black] (-0.7635211184333672,2.4651862966861966)-- ++(-2pt,-2pt) -- ++(4pt,4pt) ++(-4pt,0) -- ++(4pt,-4pt);
\draw [color=black] (0,2.4651862966861966)-- ++(-2pt,-2pt) -- ++(4pt,4pt) ++(-4pt,0) -- ++(4pt,-4pt);
\draw [color=black] (2.4651862966861966,0)-- ++(-2pt,-2pt) -- ++(4pt,4pt) ++(-4pt,0) -- ++(4pt,-4pt);
\draw [color=black] (2.4651862966861966,-0.7635211184333672)-- ++(-2pt,-2pt) -- ++(4pt,4pt) ++(-4pt,0) -- ++(4pt,-4pt);
\draw [color=black] (-1.0481507949799365,2.4651862966861966)-- ++(-2pt,-2pt) -- ++(4pt,4pt) ++(-4pt,0) -- ++(4pt,-4pt);
\draw [color=black] (-1.0481507949799365,1.9189859472289947)-- ++(-2pt,-2pt) -- ++(4pt,4pt) ++(-4pt,0) -- ++(4pt,-4pt);
\draw [color=black] (1.9189859472289947,-1.0481507949799365)-- ++(-2pt,-2pt) -- ++(4pt,4pt) ++(-4pt,0) -- ++(4pt,-4pt);
\draw [color=black] (2.4651862966861966,-1.0481507949799365)-- ++(-2pt,-2pt) -- ++(4pt,4pt) ++(-4pt,0) -- ++(4pt,-4pt);
\draw [fill=uuuuuu] (-1.0481507949799365,0) circle (2pt);
\draw [fill=uuuuuu] (-0.7635211184333671,0) circle (2pt);
\end{scriptsize}
\end{tikzpicture}
\caption{The staircase $S_{11}$ represented in the plane, with the bottom left corner of $R_5$ set as the origin.}
\label{fig:coordinates_staircase}
\end{figure}

\paragraph{\textbf{The Veech group of $S_n$.}}
Another consequence of Proposition \ref{prop:cylinders} is that there is an affine diffeomorphism $\varphi_T$ corresponding to twisting once in every horizontal cylinder. In fact, the group of orientation-preserving affine diffeomorphisms of $S_n$ is generated by $\varphi_T$ and the affine diffeomorphism $\varphi_S$ obtained by a $90^\circ$ rotation of the staircase and cutting and pasting the rotated staircase back to the original surface. These affine diffeomorphims have respective derivatives

\begin{equation*}
T := \begin{pmatrix} 1 & \lambda \\ 0 & 1 \end{pmatrix} \text{ and }
S :=\begin{pmatrix} 0 & -1 \\ 1 & 0 \end{pmatrix}
\end{equation*}
and we have

\begin{Theo}
The Veech group of the staircase model associated to $X_n$ is the Hecke group of level $n$, generated by $S$ and $T$.
\end{Theo}

In fact, it will be convenient for our purposes to add also the orientation-reversing elements. To obtain a generating set, one should add to $\varphi_T$ and $\varphi_S$ the orientation reversing diffeomorphism $\varphi_R$ given by the symmetry along the first bisection $(x,y) \leftrightarrow (y,x)$ (in the coordinates we chose for $S_n$), whose matrix is given by $R :=\begin{pmatrix} 0 & 1 \\ 1 & 0 \end{pmatrix}$.

\paragraph{\textbf{Periodic points on the double regular $n$-gons.}}
We conclude this section with a discussion on the notion of periodic point, for later use. A periodic point on a translation surface $X$ is a point whose orbit under the action of the affine group is finite. The periodic points of a Veech surface $X$ are automatically connection points, as if $P$ is a periodic point, then the marked surface $X_P$ where $P$ has been added to the set of singularities has a Veech group of finite index in $X$, hence the set of parabolic limit points coincide: every geodesic from a singularity to $P$ has a periodic direction on $X_P$, hence on $X$, and thus extends to a saddle connection on $X$. 

On the double regular $n$-gon $X_n$, the periodic points have been classified by Apisa-Saavedra-Zhang and are exactly the middle points of sides of the $n$-gon; see \cite[Theorem 1.3]{Apisa_Saavedra_Zhang}. In the staircase model, these points correspond to the center of the rectangles $R_i$, as well as the middle of the small sides of $R_1$ (the two such sides are identified) and the middle of the small sides of $R_{n-2}$ (same remark). In particular, the periodic points of the staircase model have coordinates of the form $\frac{1}{2}(x,y)$ with $x,y \in \ZZ[\lambda]$ and (indeed) do not fit in the statement of Theorem \ref{theo:non_connexion_impair}.

A major property for periodic points is that they have rational height in every cylinder, see \cite{Apisa_periodic_points}. Given a point $P$ in a cylinder $C$, we will denote by $h_C(P)$ the distance from $P$ to the boundary of $C$, and $h_C$ the height of $C$.

\begin{Def}
$P$ is said to have \emph{rational height}  (resp. \emph{irrational height}) in $C$ if $\frac{h_C(P)}{h_C} \in \QQ$ (resp. $\frac{h_C(P)}{h_C} \notin \QQ$).
\end{Def}

In fact the above is a characterization of periodic points in the staircase model of the regular $n$-gon. More precisely, we have:

\begin{Prop}\label{prop:irrational_height}
A point $P$ on the (staircase model of the) double regular $n$-gon $P$ is a periodic point if and only if it has rational height in every cylinder.
\end{Prop}

\begin{proof}
As already mentioned, periodic points have rational height in every cylinder. Hence, it suffices to show that non-periodic points have irrational height in at least one cylinder. We will work in the staircase model $S_n$. Given a non-periodic point $P$, its orbit under the action of the affine group is dense (since the orbit is not finite it must have accumulation points, and then density can be shown using twists), and in particular there is an element $Q_0 = (x,y)$ in the orbit of $P$ lying inside the region 
\[ 
Z_0 = \{ Q = (x,y) \in S_n, 0 < x< 1 \text{ and } 1-x < y < 1 - \frac{1}{\lambda}x  \} \subset R_{\frac{n-1}{2}},
\]
see Figure \ref{fig:irrational_height}. Assume that $Q_0$ has rational height in both $C_1$ and $\varphi_S(C_1)$, that is its coordinates are rational: $x, y \in \QQ$. In this case, the point $\varphi_T(Q_0)$ of coordinates $(x+ \lambda y, y)$ has height in $\varphi_S(C_2)$ given by $\frac{x+\lambda y - 1}{\lambda - 1}$, and
\begin{align*}
\frac{x+\lambda y - 1}{\lambda - 1} \in \QQ & \iff \exists p, q \in \ZZ \times \NN^{\star}, \frac{x+\lambda y - 1}{\lambda - 1} = \frac{p}{q}\\
 & \iff \exists p, q \in \ZZ \times \NN^{\star}, (x-1)q + qy \lambda = p \lambda - p
\end{align*}
and since $x$ and $y$ are rational and $\lambda \notin \QQ$, this is equivalent to
\[ \left\{ \begin{array}{ll}
x = 1 - \frac{p}{q} \\
y = \frac{p}{q}
\end{array}
\right.
\]
which is impossible because since $Q_0 \in Z_0$ we have $1-x < y$.
\end{proof}

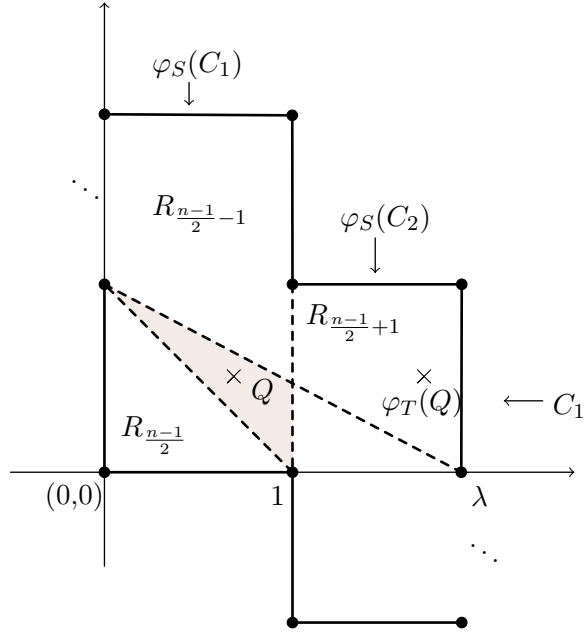
\begin{figure}[h]
\definecolor{zzttqq}{rgb}{0.6,0.2,0}
\begin{tikzpicture}[line cap=round,line join=round,>=triangle 45,x=2.5cm,y=2.5cm]
\clip(-0.5,-1) rectangle (2.75,2.75);
\fill[line width=1pt,color=zzttqq,fill=zzttqq,fill opacity=0.10000000149011612] (0,1) -- (1,0.47302251617600094) -- (1,0) -- cycle;
\draw [line width=1pt] (0,0)-- (1,0);
\draw [line width=1pt] (0,0)-- (0,1);
\draw [line width=1pt] (1,1)-- (0.998788286589887,1.8993304595392222);
\draw [line width=1pt] (0.998788286589887,1.8993304595392222)-- (0,1.9049834532258585);
\draw [line width=1pt] (1,1)-- (1.9,1);
\draw [line width=1pt] (1.9,1)-- (1.8976142827650335,0);
\draw [line width=1pt,dash pattern=on 3pt off 3pt] (1,0)-- (0,1);
\draw [line width=1pt,dash pattern=on 3pt off 3pt] (0,1)-- (1.8976142827650335,0);
\draw [line width=1pt,dash pattern=on 3pt off 3pt] (1,1)-- (1,0);
\draw (0.7261978091150156,0.55) node[anchor=north west] {$Q$};
\draw (1.42,0.5) node[anchor=north west] {$\varphi_T(Q)$};
\draw [line width=0.5pt, -to] (-0.5,0)-- (2.5,0);
\draw [line width=0.5pt, -to] (0,-0.5)-- (0,2.5);
\draw [line width=1pt] (1,0)-- (1,-0.8);
\draw [line width=1pt] (1,-0.8)-- (1.9,-0.8);
\draw (0.03415335207099173,0.35) node[anchor=north west] {$R_{\frac{n-1}{2}}$};
\draw (1.0172445433858668,0.95) node[anchor=north west] {$R_{\frac{n-1}{2}+1}$};
\draw (0.2,1.53) node[anchor=north west] {$R_{\frac{n-1}{2}-1}$};
\draw (2.3301887002077066,0.47) node[anchor=north west] {$C_1$};
\draw (0.2,2.3) node[anchor=north west] {$\varphi_S(C_1)$};
\draw (1.2,1.48) node[anchor=north west] {$\varphi_S(C_2)$};
\draw [line width=0.5pt, -to] (1.44,1.2469062727847082)-- (1.44,1.0593428218101597);
\draw [line width=0.5pt,-to] (0.45,2.074772539155129)-- (0.45,1.9518861402407697);
\draw [line width=0.5pt, -to] (2.3301887002077066,0.38)-- (2.1232221336151014,0.38);
\draw (-0.37,0) node[anchor=north west] {(0,0)};
\draw (0.83,-0.02) node[anchor=north west] {$1$};
\draw (1.8968524514044391,-0.02) node[anchor=north west] {$\lambda$};
\draw (1.8774493357863824,-0.22773051418760362) node[anchor=north west] {$\ddots$};
\draw (-0.24395797178782164,1.706113342412051) node[anchor=north west] {$\ddots$};
\begin{scriptsize}
\draw [fill=black] (0,0) circle (2pt);
\draw [fill=black] (1,0) circle (2pt);
\draw [fill=black] (0,1) circle (2pt);
\draw [fill=black] (1,1) circle (2pt);
\draw [fill=black] (0.998788286589887,1.8993304595392222) circle (2pt);
\draw [fill=black] (0,1.9049834532258585) circle (2pt);
\draw [fill=black] (1.9,1) circle (2pt);
\draw [fill=black] (1.8976142827650335,0) circle (2pt);
\draw [color=black] (0.68787,0.51)-- ++(-2.5pt,-2.5pt) -- ++(5pt,5pt) ++(-5pt,0) -- ++(5pt,-5pt);
\draw [color=black] (1.69976,0.51)-- ++(-2.5pt,-2.5pt) -- ++(5pt,5pt) ++(-5pt,0) -- ++(5pt,-5pt);
\draw [fill=black] (1,-0.8) circle (2pt);
\draw [fill=black] (1.9,-0.8) circle (2pt);
\end{scriptsize}
\end{tikzpicture}
\caption{The region $Z_0$ in the proof of Proposition \ref{prop:irrational_height}.}
\label{fig:irrational_height}
\end{figure}

\section{Hecke groups and reduction modulo two}\label{sec:Hecke} 
Given $n \geq 3$, the Hecke group $H_n$ of level $n$ is the subgroup of $PSL_2(\RR)$ generated by 
$$T = \pm \begin{pmatrix} 1 & \lambda_n \\ 0 & 1 \end{pmatrix} \text{ and } S = \pm \begin{pmatrix} 0 & -1 \\ 1 & 0 \end{pmatrix}.$$
(Recall that $\lambda_n = 2 \cos \frac{\pi}{n}$.)
These groups have been first studied by E.~Hecke \cite{Hecke_original,Hecke1983lectures}, who has shown that they are discrete subgroups of $PSL_2(\RR)$. It is often convenient to set $U = TS = \pm \begin{pmatrix} 2 \cos \frac{\pi}{n} & -1 \\ 1 & 0 \end{pmatrix}$, as then one checks that for $k \in \ZZ$,
\[ 
U^k = \frac{1}{\sin\frac{\pi}{n}} \begin{pmatrix} \sin \frac{(k+1)\pi}{n} & -\sin\frac{k\pi}{n} \\ \sin\frac{k\pi}{n} & - \sin\frac{(k-1)\pi}{n} \end{pmatrix}
\]
so that $S^2 = U^n = \pm I_2$ and it can be verified that $H_n$ is isomorphic to the free product $\ZZ / 2\ZZ * \ZZ / n\ZZ$. \newline

As a subgroup of $PSL_2(\RR)$, $H_n$ acts on the hyperbolic plane $\HH$. The quotient has finite area and $H_n$ is a lattice with a single cusp. It is a long-standing question to determine the orbit of the cusp $\infty \in \partial \HH$ under the action of $H_n$. The only known cases are $n=3,4,5,6,8,10$ and $12$, which are due to the work of several authors \cite{Rosen1954}, \cite{Leu67}, \cite{Leu74}. Such elements correspond to parabolic limit points, that is slopes of eigendirections of parabolic matrices of $H_n$, but also, by Theorem \ref{theo:Veech}, to slopes of periodic directions on the staircase model of the double regular $n$-gon. \newline

Of course, since $H_n \subset PSL_2(\ZZ[\lambda])$, we have $H_n \cdot \infty \subseteq \QQ[\lambda]$. In fact, this inclusion is an equality if and only $n=3,5$. More precisely, the work of Borho \cite{Bo73} and Borho-Rosenberger \cite{BoRo} (in German, see also the more recent work of Hanson-Merberg-Towse-Yudovina \cite{HMTY}, in English) provides an obstruction \textit{modulo two} to be in the orbit of $\infty$ for $n \geq 7$ odd. In this section, we review some of their results, which we will use in the proof of Theorem \ref{theo:non_connexion_impair} and Theorem \ref{theo:prime}.

\subsection{Obstruction modulo two: outline}\label{sec:outline_mod_2}
Let $\mathcal{O}$ be the ring of integers of $\QQ[\lambda]$, where $\lambda = \lambda_n = 2 \cos(\pi / n)$. We have $\mathcal{O} = \ZZ[\lambda]$, see e.g Proposition 2.16 of \cite{Washington}. Given $x \in \mathcal{O}$, we denote by $\overline{x}$ the reduction modulo two of $x$, which is an element of the (finite) ring $\overline{\mathcal{O}} := \mathcal{O}/2\mathcal{O}$. 

Recall that given a ring $R$, we define the projective set $\mathbb{P}^1(R)$ as the set of equivalence classes of pairs $(r,s) \in R \times R$, where $(r,s)$ and $(r',s')$ are considered equivalent if there exists an unit $u \in R^*$ such that $r = ur'$ and $s=us'$. For later use, let us denote respectively by $p$ and $\overline{p}$ the projections from $\mathcal{O}\times \mathcal{O} - \{(0,0)\}$ to $\mathbb{P}^1(\mathcal{O})$ and from $ \overline{\mathcal{O}} \times \overline{\mathcal{O}} - \{(\overline{0},\overline{0}) \}$ to $\mathbb{P}^1(\overline{\mathcal{O}})$. In our context, the Hecke group $H_n$ acts on $\mathbb{P}^1(\mathcal{O})$ since the entries of the matrices of $H_n$ belong to $\mathcal{O}$. The action of $H_n$ on $\mathbb{P}^1(\mathcal{O})$ reduces modulo two to an action of the \emph{reduced Hecke group} $\overline{H_n} \subset PSL_2(\overline{\mathcal{O}})$ on $\mathbb{P}^1(\overline{\mathcal{O}})$. More precisely, given $s = \frac{x}{y} \in \QQ[\lambda]$, with $x,y \in \mathcal{O}$ coprime, one can consider the element $[x:y]$ of $\mathbb{P}^1(\mathcal{O})$, which only depends on $s$ but not on $x$ and $y$. Hence, we have a chain of maps

\begin{align*}
\QQ[\lambda]\cup \{ \infty\} &\xrightarrow{\varphi_1}  \mathbb{P}^1(\mathcal{O})  \xrightarrow{\varphi_2}  \mathbb{P}^1(\overline{\mathcal{O}})\\
s= \frac{x}{y} &\longmapsto [x:y]  \longmapsto  [\overline{x}:\overline{y}]
\end{align*}
and the actions of $H_n$ on $\QQ[\lambda] \cup \{\infty\}$ and $\mathbb{P}^1(\mathcal{O})$ (resp. of $\overline{H}_n$ on $\mathbb{P}^1(\overline{\mathcal{O}})$) are compatible with $\varphi_1$ (resp. $\varphi_2$). From now on we will denote $\Psi := \varphi_2 \circ \varphi_1 : \QQ[\lambda] \cup \{\infty\} \to \mathbb{P}^1(\overline{\mathcal{O}})$.

In particular, if $s \in H_n \cdot \infty$, then $\Psi(s) \in \overline{H_n} \cdot \Psi(\infty)$. Now, for $n \geq 7$ odd, with the exception of $n=9$, it is a consequence of the work of Borho \cite{Bo73} (see also \cite{HMTY}) that the orbit $\overline{H_n} \cdot \Psi(\infty)$ is strictly contained in the finite set $\mathbb{P}^1(\overline{\mathcal{O}})$: this provides an obstruction for an element to be in $H_n \cdot \infty$.

\begin{Rema}
The fact that the reduction modulo two of the orbit is strictly contained in $\mathbb{P}^1(\overline{\mathcal{O}})$ is specific to the reduction modulo two, as Borho shows that for every other prime number $p \neq 2$, the image of the orbit $H_n \cdot [1:0]$ under the projection $\mathbb{P}^1(\mathcal{O}) \to \mathbb{P}^1(\mathcal{O}/p\mathcal{O})$ is all of $\mathbb{P}^1(\mathcal{O}/p\mathcal{O})$. 
\end{Rema}

\subsection{Description of the orbit}\label{sec:orbit_mod_2}
Let us be more precise in the description of $\overline{H_n} \cdot \Psi(\infty)$.\newline

We can first remark that $\overline{H_n} = < \overline{S}, \overline{U} | \overline{S}^2 = \overline{U}^n = \overline{US}^2 = id>$ and in particular $\overline{H_n}$ is a dihedral group. See \cite[Theorem 3.2]{Bo73}. On the other hand, we have:

\begin{Lem}
$\mathbb{P}^1(\overline{\mathcal{O}})$ has at least $2^{\frac{1}{2} \varphi (2n)} + 1$ elements.
\end{Lem}
\begin{proof}
Since $\lambda_n$ is an algebraic number of degree $\frac{1}{2} \varphi (2n)$, there are $2^{\frac{1}{2} \varphi (2n)}$ elements in $\mathbb{P}^1(\overline{\mathcal{O}})$ of the form $[\overline{x}:\overline{1}]$. Adding the element $[\overline{1}:\overline{0}]$ which cannot be written in this form, we obtain at least $2^{\frac{1}{2} \varphi (2n)} + 1$ elements.
\end{proof}

\begin{Rema}
$\mathbb{P}^1(\overline{\mathcal{O}})$ has exactly $2^{\frac{1}{2} \varphi (2n)} + 1$ elements if and only if $\overline{\mathcal{O}}$ is a field. As remarked in \cite{HMTY}, the smallest odd integer such that $\overline{\mathcal{O}}$ is not a field is $n=17$, and in this case $\mathbb{P}^1(\overline{\mathcal{O}})$ has $289$ elements.
\end{Rema}

Now, we turn to the study of the orbit of $[\overline{1} : \overline{0}]$ under the action of $\overline{H_n}$. Namely, we have:

\begin{Prop}\cite[Proposition 8]{HMTY} \label{prop:orbit_mod_2}
Let $n \geq 3$ odd. The size of the orbit of $\Psi(\infty) = [\overline{1} : \overline{0}]$ under $\overline{H_n}$ is at most $n$.
\end{Prop}
More precisely, it is shown in \cite{HMTY} that the set $\overline{H_n} \cdot [\overline{1}:\overline{0}]$ is the same as the orbit under the action of the subgroup of $\overline{H_n}$ generated by $\overline{U}$ which has size (at most) $n$. They remark that for $0 \leq i \leq n-1$, we have

\[ \overline{SU^i} \cdot [\overline{1}:\overline{0}] = \overline{U^{n-i-1}} \cdot [\overline{1}:\overline{0}]. \]
Which comes from the fact that

\[ \overline{SU^i} = \overline{S(TS)^i} = \overline{(ST)^i S} = (\overline{S} \cdot \overline{T})^i \cdot \overline{S} = (\overline{S^{-1}} \cdot \overline{T^{-1}})^i \cdot \overline{S} = \overline{(TS)^{-1}}^i \cdot \overline{S} = \overline{U^{-i}} \cdot \overline{S}, \]
(recall that $\overline{T} = \overline{T^{-1}}$ and $\overline{S} = \overline{S^{-1}}$), and hence

\[ 
\overline{SU^i} \cdot [\overline{1}:\overline{0}] = \overline{U^{-i} S} \cdot [\overline{1}:\overline{0}] = \overline{U^{-i-1}TSS} \cdot [\overline{1}:\overline{0}] = \overline{U^{n-i-1}} \cdot [\overline{1}:\overline{0}].
\]

As a consequence, the orbit of $[\overline{1}:\overline{0}]$ under the action of $\overline{H_n}$ is the set 
\[ \{ [\overline{P_{i+1}(\lambda)} : \overline{P_i(\lambda)}] = U^i \cdot [\overline{1}:\overline{0}], 0 \leq i \leq n-1 \} \] 
where $P_0(X) = 0, P_1(X) = 1$ and $P_{i+1}(X) = XP_i(X) + P_{i-1}(X)$ for $i \geq 1$. The fact that $\overline{SU^i} \cdot [\overline{1}:\overline{0}] = \overline{U}^{n-i-1} \cdot [\overline{1}:\overline{0}]$ gives that for $0 \leq i \leq n-1$, we get $\overline{P_{n-i}(\lambda)} = \overline{P_{i}(\lambda)}$, so that $\overline{H_n} \cdot [\overline{1}:\overline{0}]$ can be expressed alternatively as
\[ \{ [\overline{P_{i+1}(\lambda)} : \overline{P_i(\lambda)}], 0 \leq i \leq \frac{n-1}{2} \} \cup \{ [\overline{P_{i-1}(\lambda)} : \overline{P_i(\lambda)}], 1 \leq i \leq \frac{n-1}{2} \}.\]
This second expression is more convenient for our purposes as the polynomials $P_i$ for $0 \leq i \leq \frac{n-1}{2}$ have degree at most $\frac{n-3}{2}$ and hence $P_i(\lambda)$ gives the minimal expression when $n$ is prime (as the extension $[\QQ[\lambda]:\QQ]$ has degree $\frac{n-1}{2}$ when $n$ is prime). This will be useful in the proof of Theorem \ref{theo:prime}

\begin{Rema}
Note that there may be redundances in the above expressions. However it is shown in \cite{HMTY} that when $\mathcal{O}$ is a field, the size of the orbit of $[\overline{1}:\overline{0}]$ is exactly $n$.
\end{Rema}


Using the estimate $\varphi(m) \geq \sqrt{m/2}$ for $m \in \NN^*$ (which is true for powers of prime numbers and extends to any integer using its decomposition into prime numbers), one easily shows that $2^{\frac{1}{2} \varphi(2n)}+1 > n$ for $n \geq 256$, and it is easily checked numerically that the result still holds for $n=7$ as well as $n \geq 11$. As a consequence, unless $n=9$ we can always find elements of $\mathbb{P}^1(\overline{\mathcal{O}})$ not belonging to the orbit of $[\overline{1}:\overline{0}]$:

\begin{equation}\label{eq:orbit_mod2}
\overline{H_n} \cdot [\overline{1}:\overline{0}] \subsetneq \mathbb{P}^1(\overline{\mathcal{O}}).
\end{equation}


One could ask if this obstruction to be in the orbit of $\infty$ is the only obstruction. In other words, can we decide whether a given $s \in \QQ[\lambda]$ belongs to $H_n \cdot \infty$ just by looking at its reduction modulo two? For $n=9$, the answer is negative as $\overline{H}_9 \cdot [\overline{1}:\overline{0}] = \mathbb{P}^1(\overline{\mathcal{O}})$ whereas $H_9 \cdot [1:0] \neq \mathbb{P}^1(\mathcal{O})$ (see e.g. \cite{AS09, Bo20}), but it seems to be the case for $n=7$. More precisely:

\begin{Conj}\label{conj:mod2hepta}
An element $s = \frac{x}{y} \in \QQ[\lambda_7]$ belongs to $H_7 \cdot \infty$ if and only if $[\overline{x}:\overline{y}] \in \mathbb{P}^1(\overline{\mathcal{O}})$ is in the orbit of $[\overline{1}:\overline{0}]$ under the action of $\overline{H_7}$.

In other words,
\begin{equation*}
H_7 \cdot \infty = \Psi^{-1}(\overline{H}_7 \cdot \Psi(\infty))
\end{equation*}
\end{Conj}

Note that if this conjecture were true, it would give a characterization of the connection points on the double heptagon. Namely:
\begin{Conj}\label{conj:connexion_points_heptagon}[See Conjecture \ref{conj:connexion_points_heptagon_first} from the introduction]
There exist non-periodic connection points on the double heptagon. More precisely, in the associated staircase model, let $P$ be a point with coordinates $\frac{1}{N}(x,y)$ with $N \in \NN^*$ and $x,y \in \ZZ[\lambda]$ which are not both divisible by a common divisor of $N$. Then $P$ is a connection point if and only if $N$ is even and $[\overline{x}:\overline{y}]$ is in the orbit of $[\overline{1}:\overline{0}]$ under the action of the reduced Hecke group $\overline{H_7}$.
\end{Conj}
\begin{proof}[Proof that Conjecture \ref{conj:mod2hepta} implies Conjecture \ref{conj:connexion_points_heptagon}]
If $P$ has coordinates of the form $\frac{1}{N}(x,y)$ with $x,y \in \ZZ[\lambda]$ and $N$ is odd, it will be a consequence of Theorem \ref{theo:non_connexion_impair} that $P$ is not a connection point.\newline
Now, if $P$ has coordinates of the form $(\frac{x}{2k},\frac{y}{2k})$ with $x,y \in \ZZ[\lambda]$, where at least one of $x,y$ is not divisible by two, then every separatrix passing through $P$ will have an holonomy vector of the form $(\frac{x}{2k} + x', \frac{y}{2k} + y') = (\frac{x+2kx'}{2}, \frac{y+2ky'}{2})$, with $x', y' \in \ZZ[\lambda]$, and $(\overline{x},\overline{y}) \neq (\overline{0},\overline{0})$ in $\overline{\mathcal{O}}\times \overline{\mathcal{O}}$. Hence the holonomy vector of the separatrix has slope modulo two exactly $[\overline{x}:\overline{y}]$. Hence, assuming Conjecture \Ref{conj:mod2hepta}, $P$ is a connection point if and only if $[\overline{x}:\overline{y}]$ is in the orbit of $[\overline{1}:\overline{0}]$ modulo two.
\end{proof}

\begin{Rema}
A consequence of the above is that the points with coordinates of the form $\frac{1}{2k}(x,y)$ and $[\overline{x},\overline{y}] \notin \overline{H_n} \cdot [\overline{1}:\overline{0}]$ do not lie on \emph{any} saddle connection (although their coordinates lie in the trace field).
\end{Rema}

\section{Proof of Theorem \ref{theo:non_connexion_impair}}\label{sec:non_connexion_impair}

We are now ready to prove Theorem \ref{theo:non_connexion_impair}. We proceed as follows: we let $\OP$ be the orbit of $P_0$ under the action of the group of affine diffeomorphisms of $S_n$ (we do not assume here the diffeomorphisms to be orientation-preserving). The property of being or not a connection point is invariant along the orbit. 

Then, we consider the coordinate system in $S_n$ given by the planar model of $S_n$ (as in Figure \ref{fig:coordinates_staircase} for $n=11$) and where the bottom left corner of $R_{\frac{n-1}{2}}$ is set to the origin. In particular, if $P_0$ has coordinates in $\frac{1}{N} (\ZZ[\lambda] \times \ZZ[\lambda])$, then every element of its orbit $\OP$ has coordinates in $\frac{1}{N} (\ZZ[\lambda] \times \ZZ[\lambda])$ because the action of an affine diffeomorphism is given first by (maybe, the action of the orientation-reversing diffeomorphism $\varphi_R$, and) the action of a matrix $M \in H_n \subset PSL_2(\ZZ[\lambda])$ on $\RR^2$, then a translation by an element of $\ZZ[\lambda] \times \ZZ[\lambda]$ to reduce the coordinates back to the standard model, and both operations preserve the set $\frac{1}{N} (\ZZ[\lambda] \times \ZZ[\lambda])$. In particular, we can consider the reduction modulo two of $x,y$ and set, for every $(u,v) \in \overline{\mathcal{O}} \times \overline{\mathcal{O}}$, 
\[ \OP_{(u,v)} := \{ Q=\frac{1}{N}(x,y) \in \OP | \overline{x}=u \text{ and } \overline{y}=v \}. \]
In words, $\OP_{(u,v)}$ is the set of elements of the orbit of $P_0$ whose reduction modulo two of the coordinates (multiplied by $N$ so that it is an element of $\ZZ[\lambda]$) is $(u,v)$. We show:

\begin{Prop}\label{prop:dense_reduction}
For every $(u,v) \in \overline{\mathcal{O}} \times \overline{\mathcal{O}}$, $\OP_{(u,v)}$ is non-empty and, more, it is dense in $S_n$.
\end{Prop}

\begin{Rema}
Although in every coordinate system where the singularity is represented by an element of $\ZZ[\lambda] \times \ZZ[\lambda]$, it is true that the coordinates of a point $Q \in \OP$ lies in $\frac{1}{N}(\ZZ[\lambda]\times \ZZ[\lambda])$, the reduction modulo two of the coordinates may change. This is why we choose once and for all standard coordinates to work with.
\end{Rema}

As a consequence of both Proposition \ref{prop:dense_reduction} and Equation~\eqref{eq:orbit_mod2}, there exist an element $Q \in \OP$ whose coordinates are of the form $\frac{1}{N}(u_0,v_0)$ and where 
\begin{itemize}
\item $(\overline{u_0}, \overline{v_0}) \notin \overline{p}^{-1}( \overline{H_n} \cdot [\overline{1}:\overline{0}]) \cup \{(0,0)\}$
\item $Q$ belong to the rectangle $R_{\frac{n-1}{2}}$ of $S_n$.
\end{itemize}
(We recall that $\overline{p}: \overline{\mathcal{O}} \times \overline{\mathcal{O}} - \{(\overline{0},\overline{0}) \} \to \mathbb{P}^1(\overline{\mathcal{O}})$ is the projective reduction.) In particular, the separatrix from $(0,0)$ to $Q$ has a direction which reduces in $\mathbb{P}^1(\overline{O})$ to an element that is not in the orbit of $[\overline{1}:\overline{0}]$, and hence this separatrix does not have a periodic direction and does not extend to a saddle connection: $Q$, and thus $P$, is not a connection point.\newline

We are left to prove Proposition \ref{prop:dense_reduction}. The main technical step is given by the following 
\begin{Prop}\label{prop:density_one_reduction}
Let $(u,v) \in \overline{\mathcal{O}} \times \overline{\mathcal{O}}$. Assume $\OP_{(u,v)}$ is non-empty. Then it is dense in $S_n$.
\end{Prop}

\begin{proof}
Let $P \in \OP_{(u,v)}$, and assume that $P$ belongs to the rectangle $R_j$. Using twists along cylinders, we will show that $\OP_{(u,v)}$ is dense in $R_j$ as well as the adjacent rectangles $R_{j-1}$ and $R_{j+1}$.\newline

\paragraph{\textbf{First step: Density on a line.}} Since $P$ is not a periodic point, there exists a cylinder $C$ such that $P$ has irrational height in $C$ (Proposition \ref{prop:irrational_height}). We let $(\omega_C^x,\omega_C^y)$ be the holonomy vector of a core curve in the cylinder, $h_C$ be the height of the cylinder $C$, and $h_C(P)$ be the distance from $P$ to the boundary of the cylinder $C$. This cylinder $C$ is the image of one of the horizontal cylinders $C_i$ by a matrix $M \in H_n$. In particular we have $(\omega_C^x, \omega_C^y) = M\cdot(\omega_i,0) \in \ZZ[\lambda]\times \ZZ[\lambda]$. Further, the point $P$ is the image of a point $P' \in C_i$ by the affine diffeomorphism $\varphi_M$ associated to $M$. The coordinates of $P'$ lie in $\frac{1}{N}(\ZZ[\lambda] \times \ZZ[\lambda])$ and because the action of $M$ on $\RR^2$ is a linear transformation, we have
\[  \frac{h_{C_i}(P')}{h_i} =  \frac{h_{C}(P)}{h_C} \]
But we also have,
\[ \frac{h_{C_i}(P')}{h_i}\omega_i = h_{C_i}(P') \lambda \in \frac{1}{N}\ZZ[\lambda]. \]
Because $h_{C_i}(P') \in \frac{1}{N}\ZZ[\lambda]$, as it is the $y$-coordinate of $P'$ (belonging to $\frac{1}{N}\ZZ[\lambda]$) minus the $y$-coordinate of the bottom-left vertex of the rectangle $R_{j} \cup R_{j+1}$ corresponding to $C_i$ (belonging to $\ZZ[\lambda]$). In particular, we conclude
\begin{equation}\label{eq:1/NZ}
\frac{h_C(P)}{h_C}(\omega_C^x,\omega_C^y) = M \cdot (\frac{h_{C_i}(P')}{h_i}\omega_i,0)\in \frac{1}{N} (\ZZ[\lambda] \times \ZZ[\lambda])
\end{equation}

Finally, the surface is decomposed into cylinders in the direction of $C$, and all cylinders have the same moduli: there is a affine diffeomorphism $T_C$ of $S_n$ given by a simple twist in every cylinder in the cylinder decomposition in the direction of $C$.\newline

To obtain the coordinates of the image of an element of $C$ by the twist $T_C$, it is convenient to first express the point in coordinates proper to $C$, and then use translations to get back to the fundamental domain of coordinates we choose for $S_n$. For this, let us choose coordinates proper to $C$ such that, in a neighborhood of $P$ inside the fundamental domain for $S_n$, these coordinates coincide with the coordinates in $S_n$. Concretely, we unfold the cylinder $C$ starting through $P$ and $R_i$. See Figure \ref{fig:coordinate_systems}.

\begin{figure}
\center
\definecolor{ccqqqq}{rgb}{0.8,0,0}
\definecolor{wwwwww}{rgb}{0.4,0.4,0.4}
\begin{tikzpicture}[line cap=round,line join=round,>=triangle 45,x=3cm,y=3cm]
\clip(-0.9608347540347759,-0.33515871824265736) rectangle (4.139620516255229,3.0457808800589166);
\fill[line width=1pt,color=wwwwww,fill=wwwwww,fill opacity=0.05] (0,1.9) -- (0,1) -- (0,0) -- (2.2005045765705487,0) -- (2.2,1) -- (1,1) -- (1,1.9) -- cycle;
\fill[line width=1pt,color=ccqqqq,fill=ccqqqq,fill opacity=0.05] (-0.18215251352418915,0.4542276147394042) -- (2.6298067509583025,2.686319573269306) -- (2.32,2.26) -- (-0.5,0) -- cycle;
\draw [line width=1pt] (0,0)-- (1,0);
\draw [line width=1pt] (0,0)-- (0,1);
\draw [line width=1pt] (1,1)-- (1,1.9);
\draw [line width=1pt] (1,1.9)-- (0,1.9);
\draw [line width=1pt] (1,1)-- (2.2,1);
\draw [line width=1pt] (2.2,1)-- (2.2005045765705487,0);
\draw [line width=1pt,dash pattern=on 3pt off 3pt] (0,1)-- (-0.2,1);
\draw [line width=1pt,dash pattern=on 3pt off 3pt] (1,0)-- (1,-0.2);
\draw [line width=1pt,dash pattern=on 3pt off 3pt] (2.2005045765705487,0)-- (2.4,0);
\draw [line width=1pt,dash pattern=on 3pt off 3pt] (0,1.9)-- (0,2.1);
\draw [line width=1pt] (-0.18215251352418915,0.4542276147394042)-- (-0.5,0);
\draw [line width=1pt] (2.6298067509583025,2.686319573269306)-- (2.32,2.26);
\draw [line width=2.1pt,color=ccqqqq] (0,0.4967025445497522)-- (1,1.288159933965406);
\draw (0.36365704736171733,0.7743936241552819) node[anchor=north west] {$P$};
\draw (0.3,1.5) node[anchor=north west] {$R_{j-1}$};
\draw (1.3570258984090873,0.5943092125618991) node[anchor=north west] {$R_{j+1}$};
\draw (0.49726806241487237,0.48974407034638645) node[anchor=north west] {$R_j$};
\draw [line width=1pt,dash pattern=on 3pt off 3pt, to-to] (-0.5588812942423702,-0.04267958733457202)-- (-0.6587576765232951,0.09472019975445);
\draw (-0.85,0.07) node[anchor=north west] {$h_C$};
\draw (2.4,2.4) node[anchor=north west] {$(\omega_C^x,\omega_C^y)$};
\draw [color=ccqqqq](0.05,1.15) node[anchor=north west] {$L_C(P)$};
\draw [line width=1pt] (-0.18215251352418915,0.4542276147394042)-- (2.6298067509583025,2.686319573269306);
\draw [line width=1pt,-to] (-0.5,0)-- (2.32,2.26);
\begin{scriptsize}
\draw [color=black] (0,0)-- ++(-2.5pt,-2.5pt) -- ++(5pt,5pt) ++(-5pt,0) -- ++(5pt,-5pt);
\draw [color=black] (1,0)-- ++(-2.5pt,-2.5pt) -- ++(5pt,5pt) ++(-5pt,0) -- ++(5pt,-5pt);
\draw [color=black] (0,1)-- ++(-2.5pt,-2.5pt) -- ++(5pt,5pt) ++(-5pt,0) -- ++(5pt,-5pt);
\draw [color=black] (1,1)-- ++(-2.5pt,-2.5pt) -- ++(5pt,5pt) ++(-5pt,0) -- ++(5pt,-5pt);
\draw [color=black] (1,1.9)-- ++(-2.5pt,-2.5pt) -- ++(5pt,5pt) ++(-5pt,0) -- ++(5pt,-5pt);
\draw [color=black] (0,1.9)-- ++(-2.5pt,-2.5pt) -- ++(5pt,5pt) ++(-5pt,0) -- ++(5pt,-5pt);
\draw [color=black] (2.2,1)-- ++(-2.5pt,-2.5pt) -- ++(5pt,5pt) ++(-5pt,0) -- ++(5pt,-5pt);
\draw [color=black] (2.2005045765705487,0)-- ++(-2.5pt,-2.5pt) -- ++(5pt,5pt) ++(-5pt,0) -- ++(5pt,-5pt);
\draw [fill=black] (0.3463287127275988,0.7688106664456019) circle (2.5pt);
\draw [color=black] (2.6298067509583025,2.686319573269306)-- ++(-2.5pt,-2.5pt) -- ++(5pt,5pt) ++(-5pt,0) -- ++(5pt,-5pt);
\draw [color=black] (2.32,2.26)-- ++(-2.5pt,-2.5pt) -- ++(5pt,5pt) ++(-5pt,0) -- ++(5pt,-5pt);
\draw [color=black] (-0.5,0)-- ++(-2.5pt,-2.5pt) -- ++(5pt,5pt) ++(-5pt,0) -- ++(5pt,-5pt);
\draw [color=black] (-0.18215251352418915,0.4542276147394042)-- ++(-2.5pt,-2.5pt) -- ++(5pt,5pt) ++(-5pt,0) -- ++(5pt,-5pt);
\end{scriptsize}
\end{tikzpicture}
\caption{The point $P$ and a cylinder $C$ in which $P$ has irrational height. We consider alternative coordinates for elements in $C$ in which the expression of the twist along $C$ is simple. These coordinates coincide with the chosen coordinates for $S_n$ for points lying at the intersection of the cylinder $C$ with $R_j \cup R_{j-1} \cup R_{j+1}$. The line $L_C(P)$ of points having the same height in $C$ as $P$ and whose coordinates coincide with the coordinates for $S_n$ is represented in bold red.}
\label{fig:coordinate_systems}
\end{figure}
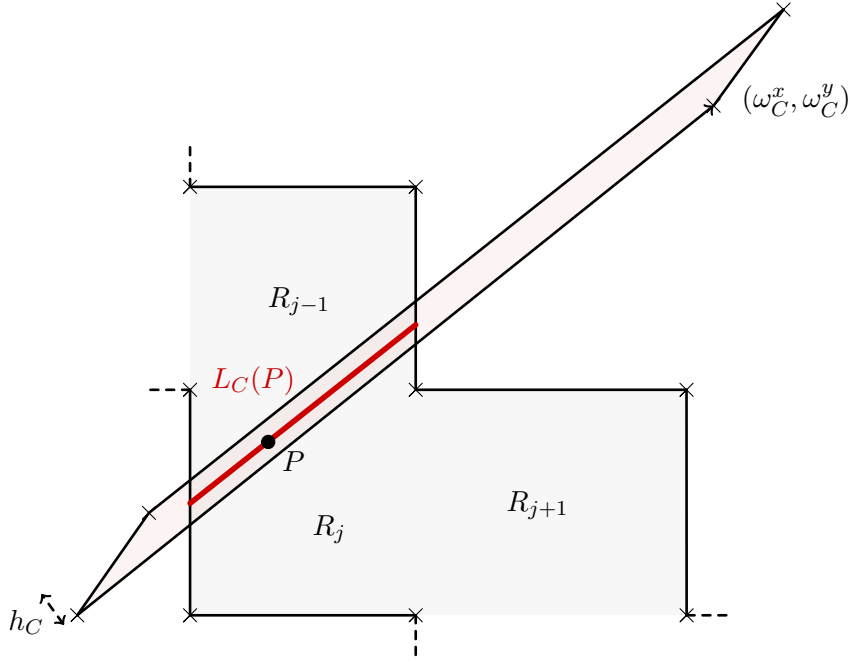

In these coordinates, the image of $P$ under $n$ twists in the cylinder $C$ is given by
\[ T_C^n(P) = P + (n \frac{h_C(P)}{h_C} - k) (\omega_C^x,\omega_C^y) \]
where $k$ is the unique integer such that $P + (n \frac{h_C(P)}{h_C} - k) (\omega_C^x,\omega_C^y)$ belongs to the fundamental domain chosen for $C$.\newline

Now, let $L_C(P)$ be the line of points at the same height as $P$ in $C$, whose coordinates in $C$ coincide with the coordinates in $S_n$. In other words, we let $I_C(P)$ be the interval containing zero of the set:
\[ \{ \eta \in \RR | P + \eta (\omega_C^x, \omega_C^y) \text{ are admissible coordinates in } S_n \} \]
and
\[ L_C(P) := \{ P + \eta (\omega_C^x, \omega_C^y) | \eta \in I_C(P) \}. \]
Let us choose now any $\eta$ in the interior of $I_C(P)$. Using the fact that $\frac{h_C(P)}{h_C} \notin \QQ$ (so that $2 \frac{h_C(P)}{h_C}$ and $2$ are rationally independent, and 
$\overline{2 \frac{h_C(P)}{h_C} \ZZ + 2\ZZ} = \RR$), for any given $\varepsilon > 0$ we can find $k,l \in \ZZ$ such that
 \[ 
|(2\frac{h_C(P)}{h}k + 2l) - \eta | \leq \varepsilon 
\]
Now, if $\varepsilon$ is chosen sufficiently small we have $(\eta-\varepsilon, \eta+\varepsilon) \subset I_C(P)$. In particular
$P + (2k \frac{h_C(P)}{h_C} + 2l) (\omega_C^x,\omega_C^y)$ have admissible coordinates in $S_n$ but also in coordinates proper to $C$, and hence the coordinates, in $S_n$, of $T_C^{2n}(P)$ are
\[ \frac{1}{N}(x_P,y_P) + (2k \frac{h_C(P)}{h_C} + 2l) (\omega_C^x,\omega_C^y) \]
where we denoted by $\frac{1}{N}(x_P,y_P) \in \frac{1}{N}(\ZZ[\lambda] \times \ZZ[\lambda])$ the coordinates of $P$ in $S_n$. In particular
\begin{itemize}
\item $T_C^{2n}(P)$ is $\varepsilon-$close to $P + \eta(\omega_C^x,\omega_C^y)$.
\item Since $\omega_C^x, \omega_C^y \in \ZZ[\lambda]$ and because of Equation~\eqref{eq:1/NZ}, we have 
\[ N \times ((2k \frac{h_C(P)}{h_C}(\omega_C^x,\omega_C^y) - 2l (\omega_C^x,\omega_C^y)) \in 2\ZZ[\lambda] \times 2\ZZ[\lambda] \]
and hence $T_C^{2n}(P) \in \OP_{(u,v)}$.
\end{itemize}

Since the above holds for any $\eta \in \mathring{I_C(P)}$ and any $\varepsilon > 0$ sufficiently small, we obtain that $\OP_{(u,v)}$ is dense in $L_C(P)$.\newline

\paragraph{\textbf{Second step : Density on a horizontal (or vertical) strip.}}
Now, we consider a cylinder $C'$ transverse to $C$, and we will assume for simplicity that it is either horizontal (if $C$ is not already horizontal) or vertical (if $C$ is horizontal). Because of the transversality assumption, the set of heights in $C'$ of elements of the line $L_C(P)$,
\[ H_{C'}(P) := \{ h_{C'}(Q) | Q \in L_C(P) \} \]
is an interval with non-empty interior. Let us show that $\OP_{(u,v)}$ is dense on the strip of elements of $C'$ whose height belongs to $H_{C'}(P)$.

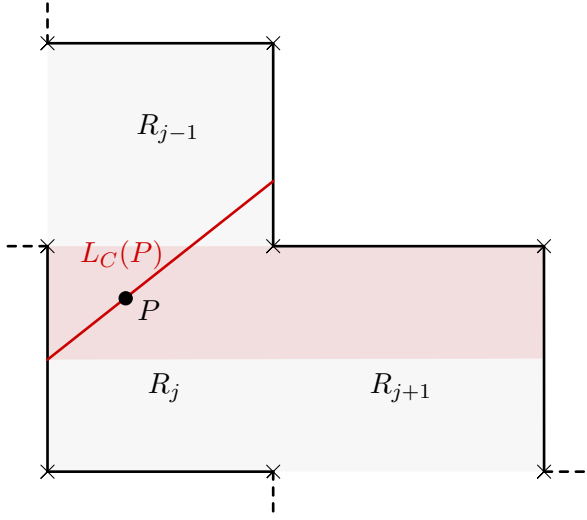
\begin{figure}
\center
\definecolor{ccqqqq}{rgb}{0.8,0,0}
\definecolor{wwwwww}{rgb}{0.4,0.4,0.4}
\begin{tikzpicture}[line cap=round,line join=round,>=triangle 45,x=3cm,y=3cm]
\clip(-0.8719547915625433,-0.5164431100333844) rectangle (3.3312372027224533,2.2697274374311585);
\fill[line width=1pt,color=wwwwww,fill=wwwwww,fill opacity=0.05] (0,1.9) -- (0,1) -- (0,0) -- (2.2005045765705487,0) -- (2.2,1) -- (1,1) -- (1,1.9) -- cycle;
\fill[line width=1pt,color=ccqqqq,fill=ccqqqq,fill opacity=0.1] (0,0.4967025445497522) -- (2.2002515976543586,0.5013687336194839) -- (2.2,1) -- (0,1) -- cycle;
\draw [line width=1pt] (0,0)-- (1,0);
\draw [line width=1pt] (0,0)-- (0,1);
\draw [line width=1pt] (1,1)-- (1,1.9);
\draw [line width=1pt] (1,1.9)-- (0,1.9);
\draw [line width=1pt] (1,1)-- (2.2,1);
\draw [line width=1pt] (2.2,1)-- (2.2005045765705487,0);
\draw [line width=1pt,dash pattern=on 3pt off 3pt] (0,1)-- (-0.2,1);
\draw [line width=1pt,dash pattern=on 3pt off 3pt] (1,0)-- (1,-0.2);
\draw [line width=1pt,dash pattern=on 3pt off 3pt] (2.2005045765705487,0)-- (2.4,0);
\draw [line width=1pt,dash pattern=on 3pt off 3pt] (0,1.9)-- (0,2.1);
\draw [line width=1pt,color=ccqqqq] (0,0.4967025445497522)-- (1,1.288159933965406);
\draw (0.3535772705524445,0.8048336444343369) node[anchor=north west] {$P$};
\draw (0.35,1.6282379986678444) node[anchor=north west] {$R_{j-1}$};
\draw (1.3828327133443286,0.48) node[anchor=north west] {$R_{j+1}$};
\draw (0.4,0.48) node[anchor=north west] {$R_j$};
\draw [color=ccqqqq](0.1,1.07) node[anchor=north west] {$L_C(P)$};
\begin{scriptsize}
\draw [color=black] (0,0)-- ++(-2.5pt,-2.5pt) -- ++(5pt,5pt) ++(-5pt,0) -- ++(5pt,-5pt);
\draw [color=black] (1,0)-- ++(-2.5pt,-2.5pt) -- ++(5pt,5pt) ++(-5pt,0) -- ++(5pt,-5pt);
\draw [color=black] (0,1)-- ++(-2.5pt,-2.5pt) -- ++(5pt,5pt) ++(-5pt,0) -- ++(5pt,-5pt);
\draw [color=black] (1,1)-- ++(-2.5pt,-2.5pt) -- ++(5pt,5pt) ++(-5pt,0) -- ++(5pt,-5pt);
\draw [color=black] (1,1.9)-- ++(-2.5pt,-2.5pt) -- ++(5pt,5pt) ++(-5pt,0) -- ++(5pt,-5pt);
\draw [color=black] (0,1.9)-- ++(-2.5pt,-2.5pt) -- ++(5pt,5pt) ++(-5pt,0) -- ++(5pt,-5pt);
\draw [color=black] (2.2,1)-- ++(-2.5pt,-2.5pt) -- ++(5pt,5pt) ++(-5pt,0) -- ++(5pt,-5pt);
\draw [color=black] (2.2005045765705487,0)-- ++(-2.5pt,-2.5pt) -- ++(5pt,5pt) ++(-5pt,0) -- ++(5pt,-5pt);
\draw [fill=black] (0.3463287127275988,0.7688106664456019) circle (2.5pt);
\end{scriptsize}
\end{tikzpicture}
\caption{In the first step, we show that $\OP_{(u,v)}$ is dense on $L_C(P)$ using the twist in a cylinder $C$ in which $P$ has irrational height. In the second step, we use a similar argument to show density of $\OP_{(u,v)}$ in the strip of elements whose height in $C'$ is the same as an element of $L_C(P)$.}
\label{fig:strip_extension}
\end{figure}

Let $h_0 \in H_{C'}(P)$. By definition of $H_{C'}(P)$ and density of $\OP_{(u,v)}$ on $L_C(P)$, for every $\varepsilon > 0$ and $M \in \NN^*$ we can find an element $Q \in \OP_{(u,v)}$ such that
\begin{align}
|h_{C'}(Q) - h_0| < \varepsilon \\
\frac{h_{C'}(Q)}{h_{C'}} \notin \{ \frac{k}{m} | 0 < k < m \leq M \} 
\end{align}
(All points of $\OP_{(u,v)}$ could have, \textit{a priori}, rational height in $C'$, but in this case the denominators have to be unbounded by density.) Further, as in Equation~\eqref{eq:1/NZ}, since $Q$ has coordinates in $\frac{1}{N}\ZZ[\lambda]\times\ZZ[\lambda]$ we have
\begin{equation}\label{eq:1/NZ2}
\frac{h_{C'}(Q)}{h_{C'}} (\omega_{C'}^x,\omega_{C'}^y) \in \frac{1}{N}(\ZZ[\lambda]\times\ZZ[\lambda])
\end{equation}
(In fact, as we assumed that $C'$ is either horizontal or vertical, we have $(\omega_{C'}^x, \omega_{C'}^y) = (\omega_i,0)$ or $(0,\omega_i)$ for some $i$).

With this choice of $Q$, the set $A = 2 \frac{h_{C'}(Q)}{h_{C'}} \ZZ + 2 \ZZ$ is $\frac{2}{M}-$dense. As in the first step, for any $\eta$ such that $Q+ \eta(\omega_{C'}^x,\omega_{C'}^y)$ has admissible coordinates for $S_n$, and is at distance at least $\frac{2}
{M}$ away from the boundary, we can find $k,l \in \ZZ$ such that 

\[ | 2 \frac{h_{C'}(Q)}{h_{C'}} k + 2l - \eta | \leq \frac{2}{M} \]
and hence
\[ T_{C'}^{2k}(Q) = Q + 2(\frac{h_{C'}(Q)}{h_{C'}} k + l) (\omega_{C'}^x,\omega_{C'}^y) \]
So that
\begin{itemize}
\item $T_{C'}^{2k}(Q)$ is $\frac{2}{M}-$close to $Q + \eta (\omega_{C'}^x,\omega_{C'}^y)$, and hence it is $\frac{2}{M} + \varepsilon$-close to the point of height $h_0$ at the vertical (resp. horizontal) of $Q + \eta (\omega_{C'}^x,\omega_{C'}^y)$.
\item Using that $\omega_{C'}^x,\omega_{C'}^y \in \ZZ[\lambda]$ and Equation~\eqref{eq:1/NZ2}, we obtain
\[ N \times (2(\frac{h_{C'}(Q)}{h_{C'}} k + l) (\omega_{C'}^x,\omega_{C'}^y)) \in 2\ZZ[\lambda] \times 2\ZZ[\lambda] \]
and hence $T_{C'}^{2k}(Q) \in \OP_{(u,v)}$.
\end{itemize}
Since this holds for any $\varepsilon > 0$ sufficiently small, $M \in \NN$ sufficiently big, and any admissible $\eta$, we obtain that $\OP_{(u,v)}$ is dense on the strip of elements of $C'$ whose height belongs to $H_{C'}(P)$.\newline

\paragraph{\textbf{Third step: Conclusion.}} Now, with the same argument as in the second step, we can choose any line inside the horizontal (resp. vertical) strip in which we know that $\OP_{(u,v)}$ is dense, consider the transverse vertical (resp. horizontal) cylinders, and apply the same argument to obtain that $\OP_{(u,v)}$ is dense on the rectangles $R_{j-1}$, $R_j$ and $R_{j+1}$. By induction, we then obtain density in every rectangle, and hence on all of $S_n$.
\end{proof}

Now we use Proposition \ref{prop:density_one_reduction} to prove Proposition \ref{prop:dense_reduction}.

\begin{proof}[Proof of Proposition \ref{prop:dense_reduction}]
Let $\pi_2(\OP)$ be the set of $(u,v) \in \overline{\mathcal{O}} \times \overline{\mathcal{O}}$ for which $\OP_{(u,v)} \neq \varnothing$. We have:
\begin{enumerate}
\item $\pi_2(\mathcal{P})$ is stable under reversal $(u,v) \mapsto (v,u)$, as we can apply $\varphi_R$ to an element of $\OP_{(u,v)}$ to obtain an element of $\OP_{(v,u)}$.
\item $\pi_2(\mathcal{P})$ is stable by adding $(\overline{0},\overline{1})$ (and hence $(\overline{1},\overline{0})$): given $(u,v) \in \pi_2(\mathcal{P})$, we know from Proposition \ref{prop:density_one_reduction} that there exists a representative element $Q = \frac{1}{N}(x,y) \in \OP_{(u,v)}$ in the rectangle $R_{\frac{n-1}{2}}$. Using the fact that $R_{\frac{n-1}{2}}$ is a unit square and that the origin has been chosen at the bottom left of $R_{\frac{n-1}{2}}$, we obtain that the image of $Q$ under the affine diffeomorphism $\varphi_S$ of derivative $S$ has coordinates $\frac{1}{N}(N-y,x)$ and if we apply the reversal $\varphi_R$ we obtain the element $\frac{1}{N}(x,N-y)$, and we have $(\overline{x},\overline{N-y}) = (\overline{x}, \overline{y}+\overline{1}) = (u,v+\overline{1})$ because $N$ is odd, as required.
\item For every $1 \leq i \leq \frac{n-1}{2}$, $\pi_2(\mathcal{P})$ is stable by the addition of $(\overline{\omega_i},0)$ (we recall that $\omega_i$ is the width of the $i^{th}$ horizontal cylinder). As before, from Proposition \ref{prop:density_one_reduction} we can choose an element of $\OP_{(u,v)}$ which lies inside the diagonal strip of the $i^{th}$ horizontal cylinder, that is the set of points of $C_i$ such that $\varphi_T^2(x,y) = (x + 2y \lambda - \omega_i, y)$, see Figure \ref{fig:diagonal_strip}. In particular, if $P \in \OP_{(u,v)}$ belongs to this diagonal strip, then the image of $P$ under the double horizontal twist $\varphi_{T}^2(P)$ lies in $\OP_{(u + \overline{\omega_i},v)}$ (here we again use that $N$ is odd).
\end{enumerate}

\begin{figure}[h]
\center
\definecolor{zzttqq}{rgb}{0.6,0.2,0}
\begin{tikzpicture}[line cap=round,line join=round,>=triangle 45,x=1cm,y=1cm, scale = 1.5]
\clip(-0.7231742469242228,-0.5262708539357027) rectangle (3.5343853230546265,2.2959383370981246);
\fill[line width=1pt,color=zzttqq,fill=zzttqq,fill opacity=0.10000000149011612] (0,2) -- (0,1) -- (3,0) -- (3,1) -- cycle;
\draw [line width=1pt] (0,0)-- (3,0);
\draw [line width=1pt] (3,0)-- (3,2);
\draw [line width=1pt] (3,2)-- (0,2);
\draw [line width=1pt] (0,2)-- (0,0);
\draw [line width=1pt] (3,0)-- (0,1);
\draw [line width=1pt] (3,1)-- (0,2);
\draw [line width=1pt,dash pattern=on 3pt off 3pt,to-to] (0,-0.2)-- (3,-0.2);
\draw [line width=1pt,dash pattern=on 3pt off 3pt,to-to] (-0.2,0)-- (-0.2,2);
\draw (-0.7,1.3115595299162224) node[anchor=north west] {$h_i$};
\draw (1.3,-0.17) node[anchor=north west] {$\omega_i$};
\begin{scriptsize}
\draw [color=black] (0,0)-- ++(-2pt,-2pt) -- ++(4pt,4pt) ++(-4pt,0) -- ++(4pt,-4pt);
\draw [color=black] (3,0)-- ++(-2pt,-2pt) -- ++(4pt,4pt) ++(-4pt,0) -- ++(4pt,-4pt);
\draw [color=black] (3,2)-- ++(-2pt,-2pt) -- ++(4pt,4pt) ++(-4pt,0) -- ++(4pt,-4pt);
\draw [color=black] (0,2)-- ++(-2pt,-2pt) -- ++(4pt,4pt) ++(-4pt,0) -- ++(4pt,-4pt);
\end{scriptsize}
\end{tikzpicture}
\caption{The diagonal strip of the $i^{th}$ horizontal cylinder $C_i$ is the set of points of $C_i$ such that $\varphi_T^2(x,y) = (x + \frac{y \omega_i}{h_i} - \omega_i, y)$.}
\label{fig:diagonal_strip}
\end{figure}
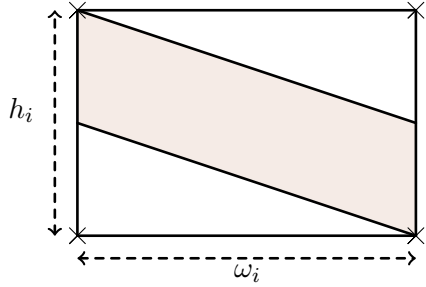

Finally, since for $1 \leq i < \frac{n-1}{2}$, $\omega_i$ can be expressed as a monic polynomial in $\lambda$ of degree $i$ (see Lemma \ref{lem:widths_cylinders}), this holds in particular for $1 \leq i < d$, where $d = \frac{1}{2} \varphi(2n) \leq \frac{n-1}{2}$ is the degree of the extension $[\QQ[\lambda]:\QQ]$. In particular, the set $\{1 \} \cup \{\omega_i, 1 \leq i \leq d-1 \}$ is a generating family of the $\ZZ$-module $\ZZ[\lambda]$ and hence, we conclude by (1), (2) and (3) that for every $(u,v) \in \mathcal{\overline{O}} \times \mathcal{\overline{O}}$ the set $\OP_{(u,v)}$ is non-empty, and hence dense in $S_n$ by Proposition \ref{prop:density_one_reduction}.
\end{proof}

This concludes the proof of Proposition \ref{prop:dense_reduction}, and hence of Theorem \ref{theo:non_connexion_impair}.

\section{Proof of Theorem \ref{theo:prime}}\label{sec:prime}
We now turn to the proof of Theorem \ref{theo:prime}. The proof relies on exhibiting an explicit element of $\mathbb{P}^1(\overline{\mathcal{O}})$ which does not belong to the orbit of $[\overline{1}:\overline{0}]$ under the action of the reduced Hecke group. Namely, 

\begin{Prop}\label{prop:lambda2}
Let $p \geq 7$ be a prime, $\lambda = 2 \cos \frac{\pi}{p}$, and $s = \frac{x}{y} \in \QQ[\lambda]$ with $x, y \in \mathcal{O}$ coprime. Assume that $\Psi(s) = [\overline{x}:\overline{y}] = [\overline{\lambda^2 +1}:\overline{1}]$. Then $s$ does not belong to the orbit of $\infty \in \partial \HH$ under the action of the Hecke group $H_p$.
\end{Prop}

The proof Proposition \ref{prop:lambda2} uses arithmetics in $\QQ[\lambda]$: the fact that $p$ is prime allows for a convenient expression of the minimal polynomial of $2 \cos \frac{\pi}{p}$ as well as the orbit of $[\overline{1}:\overline{0}]$. We then show that the direction of the separatrix in the statement of Theorem \ref{theo:prime} reduces to (an element of the orbit of) $[\overline{\lambda^2+1},\overline{1}]$. This second step is independent of the fact that $n$ is a prime number. In fact, although Proposition \ref{prop:lambda2} does not hold for $n=9$ and $n=15$, it could be checked numerically that it holds for every odd $17 \leq n \leq 201$ (we did not perform the tests after $201$), and so Theorem \ref{theo:prime} is also true for these values of $n$, independantly of $n$ being prime or not.

\subsection{Preliminary : the minimal polynomial of $\lambda$.} We start by providing a lemma that will turn useful in the proof of Proposition \ref{prop:lambda2}.

\begin{Lem}\label{lem:polynome_cyclotomique}
If $n \geq 3$ is a prime number, then the minimal polynomial $P_{\mu}$ of $\lambda_n$ is a monic polynomial of degree $\frac{n-1}{2}$, the coefficient of degree $\frac{n-3}{2}$ is $-1$ and the coefficient of degree $\frac{n-5}{2}$ is $-\frac{n-3}{2}$.
\end{Lem} 
The other coefficients will not be needed here.
\begin{proof}
Let us first recall that since $\lambda_n = e^{2i\pi/2n} + e^{-2i\pi/2n}$, the extension $\QQ[\lambda_n]$ over $\QQ$ has degree $\frac{1}{2} \varphi (2n)$ and the minimal polynomial $P_{\mu}$ of $\lambda_n$ satisfies 
\[ t^{\frac{1}{2} \varphi (2n)} P_{\mu} (t + t^{-1}) = \Phi_{2n}(t)
 \]
where $\Phi_n$ denotes the $n^{th}$ cyclotomic polynomial. In particular, when $n$ is a prime number, we have $\varphi(2n) = n-1$ and $\Phi_{2n}(X) = \frac{X^n+1}{X+1} = \sum_{k=0}^{n-1} (-1)^k X^k$. As such, $P_{\mu}$ is a monic polynomial of degree $d :=\frac{n-1}{2}$. Further, if we write $P_{\mu}(X) = X^{\frac{n-1}{2}} + \alpha X^{\frac{n-3}{2}} + \beta X^{\frac{n-5}{2}} + \cdots$, the above relation gives
\[
 t^{n-1} + \alpha t^{n-2} + \left(\frac{n-1}{2} +\beta \right) t^{n-3} + \cdots = t^{n-1} - t^{n-2} + t^{n-3} + \cdots 
\]
so that $\alpha = -1$ and $\beta = 1- \frac{n-1}{2} = -\frac{n-3}{2}$.
\end{proof}

\subsection{Proof of Proposition \ref{prop:lambda2}}\label{sec:proof_orbit}
It suffices to show that for any prime $p \geq 7$, setting $\lambda = \lambda_p = 2 \cos \frac{\pi}{p}$, the element $[\overline{\lambda^2 + 1} : \overline{1}] \in \mathbb{P}^1(\overline{\mathcal{O}})$ is not in the orbit of $[\overline{1}:\overline{0}]$ under the action of $\overline{H}_p$. \newline

We recall (see Proposition \ref{prop:orbit_mod_2} and the comments below) that the orbit of $[\overline{1}:\overline{0}]$ is given by
\[ \{ [\overline{P_i(\lambda)} : \overline{P_{i + 1}(\lambda)}], 0 \leq i \leq \frac{p-1}{2} \} \cup \{ [\overline{P_i(\lambda)} : \overline{P_{i -1}(\lambda)}], 1 \leq i \leq \frac{p-1}{2} \},\]
where $P_0(X) = 0, P_1(X) = 1$ and $P_{i+1}(X) = XP_i(X) + P_{i-1}(X)$ for $i \geq 1$ (we also have $P_{\frac{p+1}{2}}(X) = P_{\frac{p-1}{2}}(X)$). In particular, it is easily shown by induction that
\begin{Lem}\label{lem:properties_Pi}
For $1 \leq i \leq \frac{p-1}{2}$,
\begin{itemize}
\item $P_i$ is monic and of degree $i-1$.
\item $P_i$ has non trivial coefficients of degree $k$ only if $k \equiv i-1 \mod 2$,
\item (for $i \geq 3$) the coefficient of degree $i-3$ of $P_i$ is $i-2$.
\end{itemize}
\end{Lem}
Further, recall that since $p$ is prime the degree of the extension $[\QQ[\lambda] : \QQ]$ is $\frac{p-1}{2}$ and hence $P_i(\lambda)$ gives the minimal expression in $\ZZ[\lambda]$. \newline

Given $0 \leq i \leq \frac{p-1}{2}$, we first notice that if $\overline{P_{i+1}(\lambda)}$ (resp. $\overline{P_{i-1}(\lambda)}$, for $i \neq 0$) is not an unit of $\overline{\mathcal{O}}$, then we directly get
\[ [\overline{P_i(\lambda)} : \overline{P_{i \pm 1}(\lambda)}] \neq [\overline{\lambda^2 +1}:\overline{1}] \]
Otherwise, we can write
\[  [\overline{\lambda^2 +1}:\overline{1}] = [\overline{(1+\lambda^2) P_{i\pm1}(\lambda)} : \overline{P_{i \pm 1}(\lambda)}] \]
and we hence need to show:
\[ \overline{(1+\lambda^2)P_{i\pm 1}(\lambda)} \neq \overline{P_i(\lambda)}. \]

For this, we will write both $\overline{(1+\lambda^2)P_{i\pm 1}(\lambda)}$ and $\overline{P_i(\lambda)}$ in the basis $\overline{1}, \overline{\lambda}, \dots, \overline{\lambda^{\frac{p-3}{2}}}$ of $\overline{\mathcal{O}}$.\newline

First, if $i \leq \frac{p-3}{2}$, then the degree of $P_{i-1}(X) (X^2 + 1)$ is $i < \frac{p-1}{2}$ and hence the minimal expression of $P_{i-1}(\lambda) (\lambda^2 + 1)$ is a monic polynomial in $\lambda$ of degree $i$, whereas $P_i(\lambda)$ is a monic polynomial of degree $i-1$ in $\lambda$, and hence they are different, even after reduction modulo two. The same holds for $P_{i+1}(X) (X^2 + 1)$ for $i \leq \frac{p-7}{2}$.\newline

Hence it remains to compare:
\begin{enumerate}
\item $i = \frac{p-1}{2}$, $\overline{(1+ \lambda^2)P_{\frac{p-3}{2}}(\lambda)}$ and $\overline{P_{\frac{p-1}{2}}(\lambda)}$,
\item $i = \frac{p-5}{2}$, $\overline{(1+ \lambda^2)P_{\frac{p-3}{2}}(\lambda)}$ and $\overline{P_{\frac{p-5}{2}}(\lambda)}$.
\item $i = \frac{p-3}{2}$, $\overline{(1+ \lambda^2)P_{\frac{p-1}{2}}(\lambda)}$ and $\overline{P_{\frac{p-3}{2}}(\lambda)}$,
\item $i = \frac{p-1}{2}$, $\overline{(1+ \lambda^2)P_{\frac{p+1}{2}}(\lambda)}$ and $\overline{P_{\frac{p-1}{2}}(\lambda)}$,
\end{enumerate}

\begin{enumerate}
\item[(1-2)] From Lemma \ref{lem:properties_Pi}, we have
\begin{align*} (\lambda^2+1)P_{\frac{p-3}{2}}(\lambda) &= (\lambda^2+1)(\lambda^{\frac{p-5}{2}} + \frac{p-7}{2} \lambda^{\frac{p-9}{2}} + \cdots)\\
& = \lambda^{\frac{p-1}{2}} + \frac{p-5}{2} \lambda^{\frac{p-5}{2}} + \cdots
\end{align*} 
But we also have, from Lemma \ref{lem:polynome_cyclotomique}, that
\[ \lambda^{\frac{p-1}{2}} = \lambda^{\frac{p-3}{2}} + \frac{p-3}{2} \lambda^{\frac{p-5}{2}} + \cdots \]
and hence $(\lambda^2+1)P_{\frac{p-3}{2}}(\lambda)$ is a monic polynomial in $\lambda$ of degree $\frac{p-3}{2}$, whose coefficient of degree $\frac{p-5}{2}$ is $\frac{p-3}{2} + \frac{p-5}{2} = p-4$, which is odd. In particular, $\overline{(\lambda^2 + 1) P_{\frac{p-3}{2}}(\lambda)}$ must be different from $\overline{P_{\frac{p-1}{2}}(\lambda)}$ (whose coefficient of degree $\frac{p-5}{2}$ is zero) and from $\overline{P_{\frac{p-5}{2}}(\lambda)}$ (whose coefficients of both degree $\frac{p-3}{2}$ and $\frac{p-5}{2}$ are zero).
\item[(3)] Similarly, we obtain again from Lemma \ref{lem:properties_Pi} that
\begin{align*}
(\lambda^2+1)P_{\frac{p-1}{2}}(\lambda) &= \left(\lambda^2+1\right)\left(\lambda^{\frac{p-3}{2}} + \frac{p-5}{2} \lambda^{\frac{p-7}{2}}+ \cdots\right)\\
&= \lambda^{\frac{p+1}{2}} + \frac{p-3}{2} \lambda^{\frac{p-3}{2}} + \cdots
\end{align*}
But we also have 
\begin{align*} 
\lambda^{\frac{p+1}{2}} &= \lambda^{\frac{p-1}{2}} + \frac{p-3}{2} \lambda^{\frac{p-3}{2}} + \cdots \\
&= \left(\frac{p-3}{2}+1\right) \lambda^{\frac{p-3}{2}} + \cdots
\end{align*}
and hence the coefficient of degree $\frac{p-3}{2}$ of $(\lambda^2+1)P_i(\lambda)$ is $p-2$, which is odd, whereas this coefficient is zero (hence even) for $P_{\frac{p-3}{2}}(\lambda)$.
\item[(4)] Since $P_{\frac{p+1}{2}}(X) = P_{\frac{p-1}{2}}(X)$ and since we can assume that $\overline{P_{\frac{p+1}{2}}(\lambda)}$ is an unit of $\overline{\mathcal{O}}$, we cannot have $\overline{(1+\lambda^2)P_{\frac{p+1}{2}}(\lambda)} = \overline{P_{\frac{p-1}{2}}(\lambda)}$ as $\overline{\lambda^2 + 1} \neq \overline{1}$
\end{enumerate}
We finally conclude that $\overline{[\lambda^2+1:1]}$ does not belong to the orbit of $\overline{[1:0]}$, proving Proposition \ref{prop:lambda2}.

\subsection{Exhibiting a non-periodic separatrix.}\label{sec:the_separatrix}
We now prove Theorem \ref{theo:prime} by showing that the direction of the separatrix in the statement of Theorem \ref{theo:prime} has a direction which reduces modulo two to an element of the orbit of $[\overline{\lambda^2+1}:\overline{1}]$. From Proposition \ref{prop:lambda2}, the direction of such a separatrix is not periodic, and hence it does not extend to a saddle connection, proving Theorem \ref{theo:prime}.\newline

Let $n \geq 7$ odd. We recall that we consider coordinates on the double regular $n$-gon as represented on $\RR^2$ with the origin placed at the central point of one of the $n$-gons, and with a vertex at the point of coordinates $(1,0)$. We consider the separatrix starting from the point of coordinates $(\cos \frac{2\pi}{n}, \sin \frac{2\pi}{n})$ with direction 
\[ (X,Y) = \left(1+2 \cos \left(\frac{2\pi}{n}\right)\left(1 + \cos \frac{\pi}{n}\right), 2 \sin \left(\frac{2\pi}{n}\right)\left(1 - \cos \frac{\pi}{n}\right)\right). \]
See Figure \ref{fig:direction_11_gon} for $n=11$.

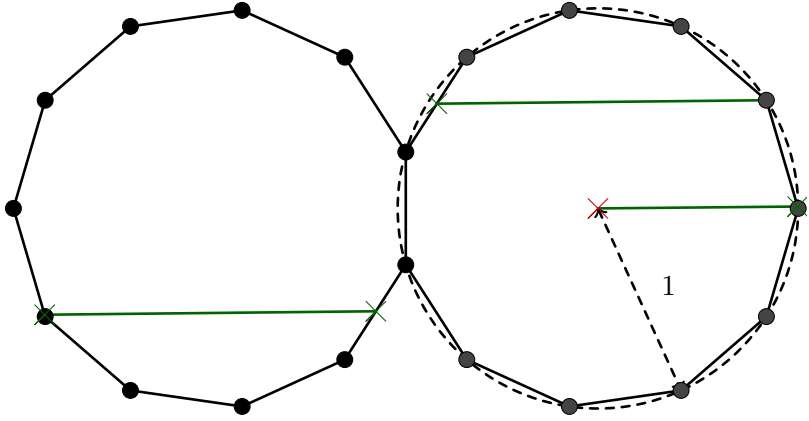
\begin{figure}[h]
\definecolor{ccqqqq}{rgb}{0.8,0,0}
\definecolor{qqwuqq}{rgb}{0,0.39215686274509803,0}
\definecolor{uuuuuu}{rgb}{0.26666666666666666,0.26666666666666666,0.26666666666666666}
\begin{tikzpicture}[line cap=round,line join=round,>=triangle 45,x=1cm,y=1cm, scale = 1.5]
\clip(-3.55,-1.35) rectangle (3.55,2.35);
\draw [line width=1pt] (0,0)-- (0,1);
\draw [line width=1pt] (0,1)-- (-0.5406408174555972,1.8412535328311805);
\draw [line width=1pt] (-0.5406408174555972,1.8412535328311805)-- (-1.450272812810115,2.256668545833067);
\draw [line width=1pt] (-1.450272812810115,2.256668545833067)-- (-2.4400942546910476,2.114353707559782);
\draw [line width=1pt] (-2.4400942546910476,2.114353707559782)-- (-3.1958438290453057,1.459492973614497);
\draw [line width=1pt] (-3.1958438290453057,1.459492973614497)-- (-3.4775763858867355,0.5);
\draw [line width=1pt] (-3.4775763858867355,0.5)-- (-3.195843829045306,-0.4594929736144966);
\draw [line width=1pt] (-3.195843829045306,-0.4594929736144966)-- (-2.440094254691048,-1.1143537075597816);
\draw [line width=1pt] (-2.440094254691048,-1.1143537075597816)-- (-1.4502728128101163,-1.256668545833067);
\draw [line width=1pt] (-1.4502728128101163,-1.256668545833067)-- (-0.5406408174555974,-0.8412535328311805);
\draw [line width=1pt] (-0.5406408174555974,-0.8412535328311805)-- (0,0);
\draw [line width=1pt] (0,0)-- (0.5406408174555972,-0.8412535328311805);
\draw [line width=1pt] (0.5406408174555972,-0.8412535328311805)-- (1.450272812810115,-1.2566685458330669);
\draw [line width=1pt] (1.450272812810115,-1.2566685458330669)-- (2.4400942546910476,-1.1143537075597818);
\draw [line width=1pt] (2.4400942546910476,-1.1143537075597818)-- (3.1958438290453057,-0.45949297361449715);
\draw [line width=1pt] (3.1958438290453057,-0.45949297361449715)-- (3.4775763858867355,0.5);
\draw [line width=1pt] (3.4775763858867355,0.5)-- (3.195843829045306,1.4594929736144966);
\draw [line width=1pt] (3.195843829045306,1.4594929736144966)-- (2.440094254691048,2.114353707559782);
\draw [line width=1pt] (2.440094254691048,2.114353707559782)-- (1.4502728128101163,2.256668545833067);
\draw [line width=1pt] (1.4502728128101163,2.256668545833067)-- (0.5406408174555974,1.8412535328311805);
\draw [line width=1pt] (0.5406408174555974,1.8412535328311805)-- (0,1);
\draw [line width=1pt,color=qqwuqq] (3.195843829045306,1.4594929736144966)-- (0.275687064782259,1.428977446237643);
\draw [line width=1pt,color=qqwuqq] (3.4727202586400665,0.5165384505944034)-- (1.7028436194446248,0.5);
\draw [line width=1pt,color=qqwuqq] (-0.2649537526733384,-0.41227608659353754)-- (-3.200699956291975,-0.44295452302009247);
\draw [line width=1pt,dash pattern=on 3pt off 3pt] (1.7028436194446248,0.5) circle (1.774732766442111cm);
\draw [line width=1pt,dash pattern=on 3pt off 3pt, to-to] (1.7028436194446248,0.5)-- (2.4400942546910476,-1.1143537075597818);
\draw (2.177777777777774,-0.009876543209877808) node[anchor=north west] {$1$};
\begin{scriptsize}
\draw [fill=black] (0,0) circle (2pt);
\draw [fill=black] (0,1) circle (2pt);
\draw [fill=black] (-0.5406408174555972,1.8412535328311805) circle (2pt);
\draw [fill=black] (-1.450272812810115,2.256668545833067) circle (2pt);
\draw [fill=black] (-2.4400942546910476,2.114353707559782) circle (2pt);
\draw [fill=black] (-3.1958438290453057,1.459492973614497) circle (2pt);
\draw [fill=black] (-3.4775763858867355,0.5) circle (2pt);
\draw [fill=black] (-3.195843829045306,-0.4594929736144966) circle (2pt);
\draw [fill=black] (-2.440094254691048,-1.1143537075597816) circle (2pt);
\draw [fill=black] (-1.4502728128101163,-1.256668545833067) circle (2pt);
\draw [fill=black] (-0.5406408174555974,-0.8412535328311805) circle (2pt);
\draw [fill=uuuuuu] (0.5406408174555972,-0.8412535328311805) circle (2pt);
\draw [fill=uuuuuu] (1.450272812810115,-1.2566685458330669) circle (2pt);
\draw [fill=uuuuuu] (2.4400942546910476,-1.1143537075597818) circle (2pt);
\draw [fill=uuuuuu] (3.1958438290453057,-0.45949297361449715) circle (2pt);
\draw [fill=uuuuuu] (3.4775763858867355,0.5) circle (2pt);
\draw [fill=uuuuuu] (3.195843829045306,1.4594929736144966) circle (2pt);
\draw [fill=uuuuuu] (2.440094254691048,2.114353707559782) circle (2pt);
\draw [fill=uuuuuu] (1.4502728128101163,2.256668545833067) circle (2pt);
\draw [fill=uuuuuu] (0.5406408174555974,1.8412535328311805) circle (2pt);
\draw [color=qqwuqq] (0.275687064782259,1.428977446237643)-- ++(-2.5pt,-2.5pt) -- ++(5pt,5pt) ++(-5pt,0) -- ++(5pt,-5pt);
\draw [color=qqwuqq] (-0.2649537526733384,-0.41227608659353754)-- ++(-2.5pt,-2.5pt) -- ++(5pt,5pt) ++(-5pt,0) -- ++(5pt,-5pt);
\draw [color=qqwuqq] (-3.200699956291975,-0.44295452302009247)-- ++(-2.5pt,-2.5pt) -- ++(5pt,5pt) ++(-5pt,0) -- ++(5pt,-5pt);
\draw [color=qqwuqq] (3.4727202586400665,0.5165384505944034)-- ++(-2.5pt,-2.5pt) -- ++(5pt,5pt) ++(-5pt,0) -- ++(5pt,-5pt);
\draw [color=ccqqqq] (1.7028436194446248,0.5)-- ++(-2.5pt,-2.5pt) -- ++(5pt,5pt) ++(-5pt,0) -- ++(5pt,-5pt);
\end{scriptsize}
\end{tikzpicture}
\caption{A separatrix on the double $11$-gon passing through one of the central points which does not extend to a saddle connection.}
\label{fig:direction_11_gon}
\end{figure}

To check that this separatrix passes through the central point of the right $n$-gon, one should notice that
\begin{align*}
X &= 1+2 \cos \left(\frac{2\pi}{n}\right) \left(1 + \cos \frac{\pi}{n}\right)\\
&= \left(\cos \frac{2\pi}{n} - \cos \frac{(n-1)\pi}{n}\right) + \left(\cos \frac{2\pi}{n} - \cos \frac{(n-3)\pi}{n}\right) +1
\end{align*}
and 
\begin{align*}
Y &= 2 \sin \left(\frac{2\pi}{n}\right) \left(1 - \cos \frac{\pi}{n}\right) \\
& = \left(\sin \frac{2\pi}{n} - \sin \frac{(n-1)\pi}{n}\right) -  \left(\sin \frac{(n-3)\pi}{n} - \sin \frac{2\pi}{n}\right)
\end{align*}

In particular, one can decompose the vector $(X,Y)$ as the sum of three vectors corresponding respectively to the holonomy vectors of
\begin{itemize}
\item the saddle connection on the right $n$-gon from the point of coordinates $\left(\cos \frac{2\pi}{n},\sin  \frac{2\pi}{n} \right)$ to the point of coordinates $\left(\cos \frac{(n-1)\pi}{n},\sin \frac{(n-1)\pi}{n} \right)$.
\item the saddle connection on the left $n$-gon from the point of coordinates $\left(\cos\frac{(n-3)\pi}{n},\sin \frac{(n-3)\pi}{n} \right)$ in the left $n$-gon (setting $(0,0)$ to be the central point of the left $n$-gon) to the point of coordinates $$\left(\cos \frac{(n+2)\pi}{n},\sin \frac{(n+2)\pi}{n} \right) = \left(-\cos \frac{2\pi}{n}, -\sin \frac{2\pi}{n} \right).$$
\item the separatrix on the right $n$-gon from the vertex of coordinates $(1,0)$ to the central point $(0,0)$.
\end{itemize}
This shows that the considered separatrix passes through the central point.\newline

We now express the direction of this separatrix in the staircase model. Recall that from the double $n$-gon to the staircase model we can apply the matrix

\[ P =\frac{1}{\sin \frac{(n-1)\pi}{2n}} \begin{pmatrix} \sin \frac{\pi}{n} & -\cos \frac{\pi}{n} + 1 \\ \sin \frac{\pi}{n} &  \cos \frac{\pi}{n} + 1 \end{pmatrix} \]

In particular, the direction $[X:Y]$ on the double $n$-gon is sent on its staircase model to 
\[ [\tilde{X}: \tilde{Y}] = \frac{\sin \frac{(n-1)\pi}{2n}}{\sin \frac{\pi}{n}} P \cdot [X:Y] = \begin{pmatrix} 1 & - \frac{\cos \left(\pi /n\right) + 1}{\sin (\pi /n)} \\ 1 & \frac{\cos \left(\pi /n\right) + 1}{\sin (\pi /n)} \end{pmatrix} [X:Y] \]
We compute:
\begin{align*}
\tilde{X} & = X -\frac{\cos \left(\pi /n\right) + 1}{\sin (\pi /n)} Y\\
& = -1 - 6 \cos \left( \frac{\pi}{n} \right) + 4 \cos^2 \left(\frac{\pi}{n}\right) + 8 \cos^3 \left(\frac{\pi}{n} \right)\\
& = -1 -3 \lambda + \lambda^2 +\lambda^3
\end{align*}
and 
\begin{align*}
\tilde{Y} & = X + \frac{\cos \left(\pi /n\right) + 1}{\sin (\pi /n)} Y\\
& = -1 + 2 \cos \left(\frac{\pi}{n}\right) + 4 \cos^2 \left(\frac{\pi}{n}\right) \\
& = -1 + \lambda + \lambda^2
\end{align*}
We want to show that the direction $[-1 -3 \lambda + \lambda^2 +\lambda^3: -1 + \lambda + \lambda^2]$ is not periodic. Up to the action of $TST^{-1}$, this is equivalent to showing that the direction $[-1 -\lambda^2: -1 - 2 \lambda]$ is not periodic. Reducing this direction modulo two gives $[\overline{1 + \lambda^2} : \overline{1}]$ which, if $n$ is a prime number, does not belong to the orbit of $[\overline{1}:\overline{0}]$ by Proposition \ref{prop:lambda2}. This implies, as required, that the direction $[\tilde{X}: \tilde{Y}]$ is not periodic on the staircase model, or equivalently that the direction $[X:Y]$ is not periodic on the double regular $n$-gon. The chosen separatrix does not extend to a saddle connection, and this finishes the proof of Theorem \ref{theo:prime}.

\begin{Rema}
Let us highlight that the only step where we need $n$ to be a prime number is to certify that $[\overline{\lambda^2 +1}:\overline{1}]$ does not belong to the orbit of $[\overline{1}:\overline{0}]$. As we already mentioned, one can check numerically that this result is true for $17 \leq n \leq 201$, independently of whether $n$ is a prime number or not. In particular the separatrix of Theorem \ref{theo:prime} does not extend to a saddle connection for these values of $n$.
\end{Rema}

\bibliographystyle{alpha}
\bibliography{bibli}
\end{document}